\documentclass[12pt,reqno]{amsart}
\usepackage{amsmath,amsthm,amsfonts,amssymb,bm}
\usepackage{fancyhdr,amscd}
\usepackage{anysize}
\usepackage{graphicx,epsfig}
\usepackage{amsthm,latexsym,amsmath,amssymb}
\usepackage{mathrsfs}
\usepackage{cite}
\usepackage{indentfirst}
 \usepackage[titletoc]{appendix}
\numberwithin{equation}{section}
\def\R{{\mathbb R}}

\newcommand{\beq}{\begin{equation}}
\newcommand{\eeq}{\end{equation}}
\newcommand{\ben}{\begin{eqnarray}}
\newcommand{\een}{\end{eqnarray}}
\newcommand{\beno}{\begin{eqnarray*}}
\newcommand{\eeno}{\end{eqnarray*}}

\newcommand{\pa}{\partial}

\newtheorem{theorem}{\textbf Theorem}[section]
\newtheorem{lemma}{\textbf Lemma}[section]

\newtheorem{cor}{\textbf Corollary}[section]

\newtheorem{prop}{\textbf Proposition}[section]
\newtheorem{defin}{\textbf Definition}[section]

\newtheorem{remark}{\it Remark}

\allowdisplaybreaks \voffset=-0.2in \numberwithin{equation}{section}

\voffset=-0.2in \numberwithin{equation}{section}\allowdisplaybreaks

\begin{document}

\title{On the striated regularity for the 2D anisotropic Boussinesq system}

\author[M. Paicu, N. Zhu]{Marius Paicu$^{1}$ and
	Ning Zhu$^{2}$}

\address{$^1$ Universit\'e de Bordeaux, Institut de Math\'ematiques de Bordeaux, F-33405 Talence Cedex, France}

\email{marius.paicu@math.u-bordeaux.fr}

\address{$^2$ School of Mathematical Sciences, Beijing Normal University, Laboratory of Mathematics and Complex Systems, Ministry of Education, Beijing 100875,  China}

\email{mathzhu1@163.com}
\begin{abstract} In this paper, we investigate the global existence and uniqueness of strong solutions to 2D  Boussinesq system  with anisotropic thermal diffusion or anisotropic viscosity  and with striated initial data.  Using  the key idea of Chemin to solve 2-D vortex patch of ideal fluid, namely the striated regularity can help to bound the gradient of the velocity, we can prove the global well-posedness  of the Boussinesq system with anisotropic thermal diffusion with initial vorticity being discontinuous across some smooth interface. In the case of an anisotropic horizontal viscosity, we can study the propagation of the striated regularity for the smooth temperature patches problem.

\end{abstract}
\maketitle

\begin{section}{introduction}
The Boussinesq system is a classical model in geophysical fluid dynamics which describes the large-scale atmospheric and oceanic flows and also play an important role in the study of Rayleigh-B\'enard convection (see \cite{Pedlosky} for example). In the present paper, we investigate the 2D anisotropic Boussinesq equations with horizontal temperature diffusion or horizontal velocity dissipation. These are derivative models from the classical Boussinesq system where the vertical dimension of the domain is very small compared with the horizontal  dimension of the domain. In this case, after rescaling the domain, the dissipation is not isotropic and we have to deal with the anisotropic problem. More precisely, we study the following system which is the Euler equations coupling with a transport-diffusion temperature equation with diffusion only in horizontal direction,
\begin{equation}
\label{system1}
\left\{
\begin{array}{cc}
\begin{split}
&\partial_t u+u\cdot\nabla u=-\nabla p+\theta e_2,~~~~~~~~~~~x\in{\mathbb{R}^2},~t>0 \\
&\partial_t \theta+u\cdot\nabla\theta-\kappa\partial_1^2\theta=0,\\
&\nabla\cdot u=0,\\
&u(0,x)=u_0(x), \theta(0,x)=\theta_0(x),\\
\end{split}
\end{array}
\right.
\end{equation}
and a system where the Navier-Stokes equations with no vertical viscosity coupling with a transport temperature equation,
\begin{equation}
\label{system2}
\left\{
\begin{array}{cc}
\begin{split}
&\partial_t u+u\cdot\nabla u-\nu\partial^2_1u=-\nabla p+\theta e_2,~~~~~~~~~~~x\in{\mathbb{R}^2},~t>0 \\
&\partial_t \theta+u\cdot\nabla\theta=0,\\
&\nabla\cdot u=0,\\
&u(0,x)=u_0(x), \theta(0,x)=\theta_0(x).\\
\end{split}
\end{array}
\right.
\end{equation}
Here $u=(u^1(x,t),u^2(x,t))$ denotes the velocity field, $p=p(x,t)$ is a scalar function denotes the pressure. $\theta=\theta(x,t)$ is a scalar representing the temperature in the content of thermal convection and the density in the modeling of geophysical fluids. $e_2=(0,1)$ is the vertical unit vector field, and the forcing term $\theta e_2$ on behalf of the buoyancy force due to the gravity field. The parameters $\kappa$ and $\nu$ denote the molecular diffusion and the viscosity respectively. These anisotropic system are important modeling dynamics of geophysical flows (see e.g. \cite{CDGG1,CDGG2,Iftimie,Paicu1}).

The general 2D anisotropic Boussinesq equations can be read as,
\begin{equation}
\label{full-boussinesq}
\left\{
\begin{array}{cc}
\begin{split}
&\partial_t u+u\cdot\nabla u-\nu_1\partial_1^2u-\nu_2\partial_2^2u=-\nabla p+\theta e_2,\ \\
&\partial_t \theta+u\cdot\nabla\theta-\kappa_1\partial_1^2\theta-\kappa_2\partial_2^2\theta=0,\\
&\nabla\cdot u=0,\\
&u(0,x)=u_0(x), \theta(0,x)=\theta_0(x).
\end{split}
\end{array}
\right.
\end{equation}
where $\nu_1$, $\nu_2$, $\kappa_1$ and $\kappa_2$ are real parameters. Systems \eqref{system1} and \eqref{system2} are two special cases for \eqref{full-boussinesq}. 
When $\nu_1=\nu_2>0$, $\kappa_1=\kappa_2>0$, the global well-posedness theory for \eqref{full-boussinesq} has been established in \cite{CD,Guo}. On the contrary, when these four parameters are zero, whether \eqref{full-boussinesq} has an unique global solution is a challenging problem and still unsolved. This system has many similarities with the classical 3D incompressible Euler equations such as the vortex stretching mechanism (which will be explained later). So it has both physical motivation and mathematical significant to investigate the intermediate cases (only partial dissipation) and some improvements has been made in the past few years. 

The global regularity for the case when $\nu_1=\nu_2>0$ and $\kappa_1=\kappa_2=0$ was proven by Chae in \cite{Chae1} and by Hou and Li in \cite{HL} with smooth initial data. Later, Abidi and Hmidi studied this system in the Besov space in \cite{AH}. The global weak solution with finite energy has been construct in \cite{HK1} and has been proved to be unique later in \cite{DP3}. For the case $\nu_1=\nu_2=0$ and $\kappa_1=\kappa_2>0$, 
Chae in \cite{Chae1} also studied the global regularity for smooth data. This result was improved by Hmidi and Keraani in \cite{HK2}, Danchin and the
first author in \cite{DP2} for rough initial data. The global well-posedness for \eqref{system1} and \eqref{system2} was considered by Danchin and the first author in \cite{DP1}, and they established the global existence and uniqueness theory. Then the global well-posedness for the anisotropic Boussinesq equations
with vertical dissipation, namely \eqref{full-boussinesq} with only $\nu_2,\kappa_2>0$, was studied by Cao and Wu in \cite{CW}. Later, Adhikaria et. al. investigated other mixed dissipation cases \cite{ACSWXY}. Other interesting recent results on the 2D anisotropic
Boussinesq equations and other related systems can be found in \cite{LLT,LT1,LPZ2,ACW1,ACW2,JL,XZ}.

Next we would like to introduce a quantity which is widely utilized in the literature we mention above. The quantity $\omega\triangleq\pa_1u_2-\pa_2u_1$ which called vorticity measures how fast the
fluid rotates. Taking $curl$ operator to the first equation of \eqref{system1}
we obtain the corresponding vorticity equation,
\begin{equation}
\label{vorticity2}
\partial_t \omega+u\cdot\nabla \omega=\partial_1\theta.
\end{equation}
Similarly, the vorticity form
of system \eqref{system2},
\begin{equation}
\label{vorticity1}
\partial_t \omega+u\cdot\nabla \omega-\partial_1^2\omega=\partial_1\theta.
\end{equation}
The forcing term $\pa_1\theta$ is the so called "vortex-stretching" term which making this system become more complex than the 2D Euler system.

Another part of our paper is devoted to study the vortex (temperature) patches problem. Before we describe this problem, we need first to introduce some notations. 
Let us denote by $\psi(\cdot,t)$ the flow associated with the vector field $u$, that is 
\begin{equation}
\label{flow-map}
\left\{
\begin{array}{cc}
\begin{split}
&\frac{d}{dt} \psi(x,t)=u(\psi(t,x),t),\\
&\psi(0,x)=x.
\end{split}
\end{array}
\right.
\end{equation}
The classical vortex patch problem is associated to the 2D Euler equations. If the initial vorticity taking the characteristic function supported in some connected bounded domain, whether the regularity of the boundary can be preserved through the evolution of the flow $\psi$? It has been proved by Chemin that the regularity of the boundary can be persisted for all the time in some H\"older class (see \cite{Chemin-perfect,Chemin4} for details). Other results about the vortex (temperature) patch problems corresponding to the Euler equations, homogeneous (inhomogeneous) Navier-Stokes equations and other fluid models can be found in \cite{GS,HH,DZ,GG-J,BC1,Danchin1,Danchin2,Hmidi1,Fanelli1,LZ1,LZ2,DZ2,DM1,DM2,PZ1,DFP} and the references therein.

In order to understand the striated regularity clearly, we need first to introduce some notations and definitions which 
will be used to describe the boundary regularity. Let $X_0$ be a vector field defined on $D_0$ (a connected bounded  domain),
$X$ is the evolution of $X_0$ along the flow $\psi$ defining as follows,
\begin{equation}
\label{X-defin}
X(x,t)\triangleq\pa_{X_0}\psi(\psi^{-1}(x,t),t),
\end{equation} 
where $\pa_{X_0}f\triangleq X_0\cdot\nabla f$ denoting the standard directional derivative.\\
Taking time derivative of \eqref{X-defin}, one can obtain $X$ satisfies the following transport equation,
\begin{equation}
\label{X-eq}
\left\{
\begin{array}{cc}
\begin{split}
&\partial_t X+u\cdot\nabla X=\partial_X u, \\
&X(0,x)=X_0(x).\\
\end{split}
\end{array}
\right.
\end{equation}
It is not hard to check that $\pa_X$ satisfies,
\begin{equation}
\label{properties-partialX}
[\pa_X, D_t]=0,
\end{equation}
where $[A, B]\triangleq AB-BA$ represents the standard commutator, and
$D_t\triangleq \pa_t+u\cdot\nabla$ denotes the material derivative.

We need also the following two definitions, which can be found in \cite{Chemin-perfect,BCD}.

\begin{defin}
	\label{Cs-class}
	Let $s>0$ and $\Omega$ be a bounded domain in $\R^d$. We say that $\Omega$ is of class $C^{s}$
	if there exists a compactly support function $f\in C^s(\R^2)$ and a neighborhood $V$ of $\pa\Omega$ such that 
	\begin{equation*}
	\pa\Omega=f^{-1}(\{0\})\cap V\quad \text{and}\quad \nabla f(x)\neq0~~\forall~x\in V.
	\end{equation*}
\end{defin}

\begin{defin}
\label{definition2.1}
A family $(X_\lambda)_{\lambda\in\Lambda}$ of vector fields over $\R^2$ is said to be non-degenerate whenever
\begin{equation*}
I(X)\triangleq\inf_{x\in\R^d}\sup_{\lambda\in\Lambda}|X_{\lambda}(x)|>0.
\end{equation*}
Let $r\in(0,1)$ and $(X_{\lambda})_{\lambda\in \Lambda}$ be a non-degenerate family of $C^r$ vector fields over $\R^2$. A bounded function $f$ is said to be in the function space $C^r_{X}$ if it satisfies
\begin{equation*}
\|f\|_{C^r_X}\triangleq \sup_{\lambda\in\Lambda}\bigg(\frac{\|f\|_{L^\infty}\|X_\lambda\|_{C^r}+\|\nabla\cdot(X_{\lambda}f)\|_{C^r_X}}{I(X)}\bigg).
\end{equation*}
\end{defin}

Next we present the main results for our paper. Since the concrete values of the constants $\kappa$ in system \eqref{system1} and $\nu$ in \eqref{system2} play no role in our discussion, for this reason, we shall assume $\kappa=\nu=1$ throughout this paper.

The main result pertaining to system \eqref{system1} can be stated as follows.
\begin{theorem}
\label{theorem-model1}
Assume $u_0\in L^2$ be a divergence-free vector field, the corresponding vorticity $\omega_0\triangleq\partial_1u_0^2-\partial_2u_0^1\in L^\infty$.
Let $(\omega_0,\theta_0)\in H^s\times H^{1+s}$ with $0<s<1.$ Then system \eqref{system1} exists a unique global solution $(u,\theta)$ satisfies
\begin{equation*}
u\in L^\infty([0,T];H^{1+s}),~\omega  \in L^\infty([0,T];L^\infty),~ \theta\in
L^\infty([0,T];H^{1+s}),~ \partial_1\theta\in
L^\infty([0,T];H^{1+s}).
\end{equation*}
Furthermore, for any non-degenerate vector field $X_0\in C^s$ such that $\pa_{X_0}\omega_0\in L^p~(2<p<\infty)$, there exists a unique global solution $X\in L^\infty([0,T];C^s)$ to equation \eqref{X-eq} and we have
\begin{equation*}
\partial_X\omega\in L^\infty([0,T];L^p),~\nabla u\in L^1([0,T];L^\infty).
\end{equation*}
\end{theorem}
As a direct application, this theorem can be used to deal with the so called "vortex patch" problem as follows.\\
For
\begin{equation}
\label{vortex-patch-initial}
\omega_0(x)=\chi_{D_0}(x)\triangleq\left\{\begin{split}
1\quad x\in D_0,\\
0\quad x\notin D_0,
\end{split}\right.
\end{equation}
where $D_0$ is a connected bounded domain, $\chi_{D_0}$ is the standard characteristic function of $D_0$. Let $\omega(x,t)=\omega^1(x,t)+\omega^2(x,t)$ where $\omega^1$ is the solution of the system
\begin{equation}
\label{omega1}
\left \{\begin{split} &\pa_t\omega^1+u\cdot \nabla \omega^1=0,\\
&\omega^1(x,0)=\omega_0(x),\end{split}
\right.
\end{equation}
and $\omega^2$ is the solution of the system
\begin{equation}
\label{omega2}
\left \{\begin{split} &\pa_t\omega^2+u\cdot \nabla \omega^2=\pa_1\theta,\\
&\omega^2(x,0)=0.\end{split}
\right.
\end{equation}
Then the main result can be stated as follows.
\begin{cor}
	\label{cor-1}
Assume $\omega_0$ defined as in \eqref{vortex-patch-initial} and $D_0$ be a connected bounded domain with its boundary $\pa D_0$ in H\"older class $C^{1+s}~(0<s<1)$. Then system \eqref{system1} exists a unique global solution satisfies the properties shows in Theorem \ref{theorem-model1}. Moreover, the solution of systems \eqref{omega1} and \eqref{omega2} satisfying
\begin{equation*}
\omega^1=\chi_{D_t},\quad \quad \omega^2\in L^\infty([0,T];C^{s}(X)),
\end{equation*}
with $D_t\triangleq\psi(D_0,t)$ and
the boundary of the domain remains in the class $C^{1+s}$.
\end{cor}

Then we present our main result pertaining to system \eqref{system2}.
\begin{theorem}
\label{theorem-model2}
Assume $u_0\in L^2$ be a divergence-free vector field, the corresponding vorticity $\omega_0\triangleq\partial_1u_0^2-\partial_2u_0^1\in L^2\cap L^\infty$.
Let $(\omega_0,\theta_0)\in H^s\times H^{\beta}$ with $\frac12<s<\beta.$ Then system \eqref{system2} exists a unique global solution $(u,\theta)$ satisfies
\begin{equation*}
u\in L^\infty([0,T];H^{1+s}),~\partial_1u\in L^2([0,T];H^{1+s}),~\nabla u \in L^1([0,T];L^\infty),~ \theta\in
L^\infty([0,T];H^{s}).
\end{equation*}
Furthermore, for any vector field $X_0\in H^s$, there exists a unique global solution $X\in L^\infty([0,T];H^s)$ to equation \eqref{X-eq}. Moreover,
$X\in L^\infty([0,T];H^{s'})$ for $s'>1$ if provided $\omega_0\in\dot{ W}^{1,p}\cap H^{s'}, \theta_0\in \dot{ W}^{1,p}\cap H^{s'}$ with $2<p<\infty$ and $X_0\in H^{s'}$.
\end{theorem}

\begin{remark}
Here we obtain the velocity $u$ is Lipschitz, which is more regular compared with the result of the paper of Danchin and the first author \cite{DP1}, where the velocity  was only Log-Lipschitz.
\end{remark}

\begin{remark}
In the critical case $s=1/2$, we can prove the global well-posedness and Lipschitz information for velocity with $\omega_0\in B^{\frac12}_{2,1}$ and $\theta_0\in H^\beta$, $1/2<\beta$. The method is much similar to our proof of Theorem \ref{theorem-model2} but the process is more complicated. In order to make our paper easy to read, we only discuss the result in Sobolev space here.
\end{remark}

\begin{remark}
We can even obtain the Lipschitz information of the velocity with initial vorticity $\omega_0$ in anisotropic Besov space $\mathcal{B}^{0,\frac12}$
through a similar idea. 
Here $\mathcal{B}^{0,\frac12}$ is the space given by the norm
\begin{equation*}
\|f\|_{\mathcal B^{0,\frac 12}}=\sum\limits_{q\in Z}2^{\frac q2}\|\Delta_q^v f\|_{L^2} \quad\text{and}\quad \Delta_q^v=\mathcal F^{-1}(\varphi(\xi_2/2^q)\hat f(\xi))
\end{equation*}
 is the dyadic bloc in the vertical Fourier variable and the definition of $\varphi(\xi)$ will be given in the next section.
\end{remark}

The above result can be used to solve the smooth "temperature patch" problem.
Defining
\begin{equation}
\label{temperature-patch-initial}
\theta^{\varepsilon}_0(x)=\chi_{D_0}*\eta_{\varepsilon}(x)=\left\{\begin{split}
&1\quad x\in D^-_{\varepsilon},\\
&0\quad x\in \R^2\setminus D^+_{\varepsilon},
\end{split}\right.
\end{equation}
where $\chi_{D_0}$ is the characteristic function of the domain $D_0$. $\eta_\varepsilon$ is the standard mollified function. $D^-_{\varepsilon}$ and $D^+_{\varepsilon}$ are two domains defined by
\begin{equation*}
\begin{split}
D^-_{\varepsilon}\triangleq\{x\in D: dist(x,\pa D_0)>\varepsilon\},
\\
D^+_{\varepsilon}\triangleq\{x\in \R^2: dist(x,\pa D_0)>\varepsilon\}.
\end{split}
\end{equation*}
Along the evolution of the fluid, the distance of $\psi(D^-_{\varepsilon},t)$
and $\psi(D^+_{\varepsilon},t)$
denoted by $d(t)$ with $d(0)=2\varepsilon$. Then the following result hold true.

\begin{cor}
	\label{cor-2}
Let $\frac12<s<1$, assume $\theta_0=\theta_0^{\varepsilon}$ defined as in \eqref{temperature-patch-initial} with $\pa D_0\in H^{1+s}$, $\omega_0\in L^\infty\cap H^s$. Then
there exists a unique solution $(u,\theta)$ to
system \eqref{system2} satisfying
the properties listed in Theorem \ref{theorem-model2}.
Furthermore, $\theta(x,t)$ satisfies the same form as $\theta_0$ that
\begin{equation*}
\theta(x,t)=\left\{\begin{split}
&1\quad x\in \psi(D^-_{\varepsilon},t), \\
&0\quad x\in \R^2\setminus \psi(D^+_{\varepsilon},t),
\end{split}\right.
\end{equation*}
and the distance $d(t)$ satisfies,
\begin{equation}
\label{distance-temperature-patch}
|d(t)|\leq 2\varepsilon e^{\int_0^t\|\nabla u(\tau)\|_{L^\infty}d\tau}.
\end{equation}
Moreover, the flow $\psi(\cdot,t)\in H^{1+s}$
and the boundary $\pa D^-_{\varepsilon}$, $\pa D^+_{\varepsilon}\in H^{1+s}$ for all $t\geq0$.
\end{cor}
\noindent
\textbf{Remark:} We can propagate higher regularity of the boundary for the temperature patch if we improve the regularity condition of the initial data.

The rest of this paper is divided into three sections and an appendix. In section 2, we provide some definitions and lemmas which will be used in the next sections. Section 3 is devoted to the study of system \eqref{system1} 
which divided into three subsections. The first one gives some regularity estimates, the second subsection shows the estimate for striated regularity and the last subsection gives the proof of Corollary \ref{cor-1}. Section 4 deals with system \eqref{system2} which is divided into five subsections unfolding similar as section 3. Finally, Appendix A provides the technical proof for some lemmas presented in the second section.

\end{section}

\begin{section}{Preparations}
In this section, we will give some definitions and lemmas which will be used in the next several sections. First we give some notations. Throughout this paper, $C$ stands for some real positive constant which may vary from line to line. $\{b_q\}$ stands for the $\ell^1$ sequence which may also different in each occurrence. $|D|\triangleq(-\Delta)^{\frac12}$ denotes the Zygmund operator which
is defined through the Fourier transform that
\begin{equation}
\widehat{|D|f}=|\xi|\widehat{f},
\end{equation}
where
\begin{equation*}
\widehat{f}\triangleq\mathcal{F}(f)=\frac{1}{(2\pi)^2}\int_{\R^2}e^{-ix\cdot\xi}f(x)~dx.
\end{equation*}
Similarly, we can define
\begin{equation}
\widehat{|D|^sf}=|\xi|^s\widehat{f},\quad\quad\widehat{|\pa_1|^sf}=|\xi_1|^s\widehat{f}.
\end{equation}

Next we present the classical Littlewood-Paley theory in $\R^d$ which plays an important role in the proof of our result. Let $\chi$ be a smooth function support on the ball $\mathcal{B}\triangleq\{\xi\in\R^d:|\xi|\leq\frac43\}$ and $\varphi$ be a smooth function support on the
ring $\mathcal{C}\triangleq\{\xi\in\R^d:\frac34\leq\xi\leq\frac83\}$ such that
\begin{equation*}
\chi(\xi)+\sum_{q\geq0}\varphi(2^{-q}\xi)=1,~~~~\text{for all}~~~~\xi\in\R^d,~\quad~~~~~~~
\sum_{q\in\mathbb{Z}}\varphi(2^{-q}\xi)=1,~~~~\text{for all}~~~~\xi\in\R^d\setminus \{0\}.
\end{equation*}
Then for every $u\in \mathcal{S'}$ (tempered distributions), we define the non-homogeneous Littlewood-Paley operators as follows,
\begin{equation*}
\begin{split}
&\quad\Delta_q u=0~\text{for}~q\leq-2,\quad\Delta_{-1}u=\chi(D)u=\mathcal{F}^{-1}(\chi(\xi)\widehat{u}(\xi)),\\
&\Delta_qu=\varphi(2^{-j}D)u=\mathcal{F}^{-1}(\varphi(2^{-j}\xi)\widehat{u}(\xi)),~~\forall~q\geq 0,\quad S_qu=\sum_{j=-1}^{q-1}\Delta_ju.
\end{split}
\end{equation*}
Next we state the definition of non-homogeneous Besov spaces through the dyadic decomposition.
\begin{defin}
For $s\in\R$ and $1\leq p,r\leq \infty$, the non-homogeneous Besov space ${B}_{p,r}^s$ is defined by
\begin{equation*}
{B}_{p,r}^s=\{f\in\mathcal{S}'; \|f\|_{{B}_{p,r}^s}<\infty\},
\end{equation*}
where
\begin{equation*}
\|f\|_{{B}_{p,r}^s}=\left \{\begin{split}&\sum_{q\geq-1}(2^{qs}\|\Delta_qf\|_{L^p}^r)^{\frac1r}\quad for~r\leq\infty,\\
&\sup_{q\geq-1}2^{qs}\|\Delta_qf\|_{L^p}\quad for~r=\infty. \end{split} \right.
\end{equation*}
\end{defin}
\noindent
We point out that when $p=r=2$, for all $s\in \mathbb{R}$, we have $B^s_{2,2}({\mathbb{R}^d})=H^s({\mathbb{R}^d})$.

 \begin{lemma}{(Bernstein inequality\cite{Chemin-perfect,BCD})}
	\label{bernstein}
	Let $k\in\mathbb{N}\cup\{0\}$, $1\leq a\leq b\leq\infty$. Assume that
	\begin{equation*}
	supp\widehat{f}\subset\big\{\xi\in{\mathbb{R}^d}: |\xi|\leq2^q\mathcal{C}\big\},
	\end{equation*}
	for some integer $q$, then there exists a constant $C_1$ such that
	\begin{equation*}
	\|\nabla^{\alpha}f\|_{L^b}\leq C_12^{q\big(k+d\big(\frac{1}{a}-\frac{1}{b}\big)\big)}\|f\|_{L^a},~~k=|\alpha|.
	\end{equation*}
	If $f$ satisfies
	\begin{equation*}
	supp\widehat{f}\subset\big\{\xi\in{\mathbb{R}^n}: |\xi|=2^q\mathcal{C}\big\},
	\end{equation*}
	for some integer $q$, then
	\begin{equation*}
	C_22^{qk}\|f\|_{L^b}\leq\|\nabla^{\alpha}f\|_{L^b}\leq C_32^{q\big(k+d\big(\frac{1}{a}-\frac{1}{b}\big)\big)}\|f\|_{L^a},~~k=|\alpha|,
	\end{equation*}
	where $C_2$ and $C_3$ are constants depending on $\alpha$, $a$ and $b$ only.
\end{lemma}

Noticing that if $u$ is a divergence-free vector field in $\R^2$, then it can be recovered from the corresponding vorticity $\omega$ by means of the following Biot-Savart law
\begin{equation}
\label{biot-savart}
u=\nabla^{\perp}\Delta^{-1}\omega.
\end{equation}
Combining the classical Calder\'on-Zygmund estimate and \eqref{biot-savart}, it can lead to the following lemma \cite{Chemin-perfect}.
\begin{lemma}
For any smooth divergence-free vector field $u$ with its vorticity $\omega\in L^p$ and $p\in(1,\infty)$, there exists a constant $C$ such that
\begin{equation}
\|\nabla u\|_{L^p}\leq C\frac{p^2}{p-1}\|\omega\|_{L^p}.
\end{equation}
\end{lemma}
The next lemma shows the H\"older estimate for transport equation, which is useful in the estimate of the striated regularity. The proof can be found in \cite{Chemin-perfect}.
\begin{lemma}
	\label{holder-transport}
Let $v$ be a smooth divergence-free vector field, $r\in(-1,1)$. Consider two functions $f\in L^{\infty}_{loc}(\R;C^{r})$ and $g\in L^{1}_{loc}(\R;C^{r})$ satisfy the transport equation
\begin{equation*}
\partial_tf+u\cdot\nabla f=g.
\end{equation*}
Then we have
\begin{equation*}
\|f(t)\|_{C^r}\leq C\|f(0)\|_{C^r}e^{C\int_0^t\|\nabla u(\tau)\|_{L^\infty}~d\tau}+C\int_0^t\|g(\tau)\|_{C^r}e^{C\int_{\tau}^t\|\nabla u(s)\|_{L^\infty}~ds}~d\tau,
\end{equation*}
and the constant $C$ depends only on $r$.
\end{lemma}

The following logarithmic inequality plays an important role in the proof of the Lipschitz information for velocity of system \eqref{system1}. The proof of this lemma can be found in \cite{Chemin-perfect,BCD}.
\begin{lemma}
\label{log-ineq-lemma}
Let $r\in(0,1)$ and $(X_{\lambda})_{\lambda\in \Lambda}$ be a non-degenerate family of $C^r$ vector fields over $\R^2$. Let $u$ be a divergence-free vector field over $\R^2$ with vorticity $\omega\in C^r_X$. Assume, in addition that $u\in L^q$ for some $q\in[1,+\infty]$ or that $\nabla u\in L^p$ for some finite p.
Then there exists a constant $C$ depending on $p$ and $r$ such that
\begin{equation}
\label{log-ineq}
\|\nabla u\|_{L^\infty}\leq
C\bigg(\min(\|u\|_{L^q},\|\omega\|_{L^p})+\|\omega\|_{L^\infty}\log\bigg(e+\frac{\|\omega\|_{C^r_X}}{\|\omega\|_{L^\infty}}\bigg)\bigg).
\end{equation}

\end{lemma}

Then we give the definition of the space $\sqrt{L}$ and $LL^{\frac12}$.
 \begin{defin}
 \label{defin-log-L}
 The space $\sqrt{L}$ stands for the space of functions $f$ in $\bigcap_{2\leq p<\infty}L^p$ such that
 \begin{equation*}
 \|f\|_{\sqrt{L}}\triangleq \sup_{p\geq2}\frac{\|f\|_{L^p}}{\sqrt{p-1}}<\infty.
 \end{equation*}
 And the space $LL^{\frac12}$ denotes by
 \begin{equation*}
 LL^{\frac12}\triangleq \big\{f\in\mathcal{S}':\|f\|_{LL^{\frac12}}\triangleq\sup_{j\geq0}\frac{\|S_jf\|_{L^\infty}}{\sqrt{j+1}}<\infty\big\}.
 \end{equation*}
 \end{defin}
 \noindent
\textbf{Remark}:~
It is not hard to check that $\sqrt L \hookrightarrow LL^{\frac12}$.

The following lemma play a significant role in the estimate of the convection term. The proof of this lemma shall be shown in the Appendix.
\begin{lemma}
	\label{lemA.1}
	Assume $u$ is a smooth divergence free vector field with $u\in L^2$, $\nabla u\in L^\infty$, $f\in H^s$ with $s\in (0,1)$, then we have
	\begin{equation}
	\begin{split}
	\label{lem1eq1}
	-\int_{\R^2}\Delta_q(u\cdot \nabla f)\Delta_q f~dx\leq&  C b_q2^{-2qs}\|\nabla u\|_{L^\infty}
	\|f\|_{H^s}^2,
	\end{split}
	\end{equation}
	with $b_q\in\ell^1$. Moreover, if $\omega,\pa_1\omega\in L^2$, $\partial_1f\in H^s$, then we have
	\begin{equation}
	\begin{split}
	\label{lem1eq}
	-\int_{\R^2}\Delta_q(u\cdot \nabla f)\Delta_q f~dx\leq&  C b_q2^{-2qs}(\|u\|_{L^2}+\|\omega\|_{L^2}+\|\partial_1\omega\|_{L^2})\\&\quad\times(\|f\|_{H^s}^2+\|f\|_{H^s}^\frac12\|\partial_1f\|_{H^s}^\frac12+\|f\|_{H^s}^\frac32\|\partial_1f\|_{H^s}^\frac12),
	\end{split}
	\end{equation}
	where $\omega$ is the corresponding vorticity of $u$.
\end{lemma}

Then we give a lemma which alerts the classical losing regularity estimate for the transport equation,
and result can be found in \cite{BCD,DP1}. For the sake of completeness, we will give the proof in the Appendix.
\begin{lemma}[Losing regularity estimate for transport equation]
	\label{lemma-losing-transport}
	Let $\rho$ satisfies the transport equation
	\begin{equation}
	\label{tran}
	\left\{
	\begin{array}{cc}
	\begin{split}
	&\partial_t \rho+u\cdot\nabla \rho=f,\\
	&\rho(0,x)=\rho_0(x),
	\end{split}
	\end{array}
	\right.
	\end{equation}
	where $\rho_0\in B^s_{2,r}$, $f\in L^1([0,T];B^s_{2,r})$ with $r\in [1,\infty]$. Here $v\in L^2$ is a divergence free
	vector field and for some $V(t)\in L^1([0,T])$, $v$ satisfies
	$$\sup_{N\geq0}\frac{\|\nabla S_N v(t)\|_{L^\infty}}{\sqrt{1+N}}\leq V(t). $$
	Then for all $s>0$, $\varepsilon\in (0,s)$ and $t\in [0,T]$, we have the following estimate,
	\begin{equation*}
	\begin{split}
	\|\rho(t)\|_{B^{s-\varepsilon}_{2,r}}&\leq C(T) \bigg(\|\rho_0\|_{B^{s}_{2,r}}+\int_0^T\|f(\tau)\|_{B^{s}_{2,r}}~d\tau\bigg)e^{\frac{C}{\eta}\int_0^TV(\tau)~d\tau},
	\end{split}
	\end{equation*}
\end{lemma}

The following Lemma gives the classical Kato-Ponce type inequality, which can be found in \cite{KP,Kato,KPV}.
\begin{lemma}
\label{kato-ponce}
Assume $s>0$ and $p\in(1,+\infty)$. Let $f$ satisfies $f\in L^{p_1}$, $\nabla f\in L^{p_1}$, $|D|^sf\in L^{p_3}$, $g$ satisfies
$|D|^{s-1} g\in L^{p_2}$, $|D|^{s} g\in L^{p_2}$, $g\in L^{p_4}$, then we have
\begin{equation}
\label{kato-ponce1}
\|[|D|^s,f]g\|_{L^p}\leq C(\|\nabla f\|_{L^{p_1}}\||D|^{s-1} g\|_{L^{p_2}}
+\||D|^s f\|_{L^{p_3}}\| g\|_{L^{p_4}}),
\end{equation}
\begin{equation}
\label{kato-ponce2}
\||D|^s(fg)\|_{L^p}\leq C(\| f\|_{L^{p_1}}\||D|^{s} g\|_{L^{p_2}}
+\||D|^s f\|_{L^{p_3}}\| g\|_{L^{p_4}}),
\end{equation}
where $p_2, p_3\in (1,+\infty)$ satisfy
\begin{equation*}
\frac1p=\frac{1}{p_1}+\frac{1}{p_2}=\frac{1}{p_3}+\frac{1}{p_4}.
\end{equation*}
\end{lemma}

\end{section}

\begin{section}{The Case of Horizontal Diffusivity}
This section is devoted to deal with the first model \eqref{system1}. At the beginning, we will give some regularity estimates for $(\omega, \theta)$ in the first subsection. Then we will exam the H\"older estimate of $X$ and prove Corollary \ref{cor-1}.

\subsection{A priori estimates for $\omega$ and $\theta$}
\indent\\
Before we give the regularity estimate for $(\omega,\theta)$, we need first recall the following existence and uniqueness result in \cite{DP1} about system \eqref{system1}.
\begin{theorem}
	\label{result-DP-model1}
Let $1<s<\frac32$ and $\theta_0\in H^1$ such that $|\partial_1|^s\theta_0\in L^2$. Let $u_0\in H^1$ be a divergence-free vector field and the corresponding vorticity $\omega_0$ in $L^\infty$. Then system \eqref{system1} with initial data $(\theta_0, u_0)$
admits a global unique solution $(\theta, u)$ in $C_{w}(\R_+;H^1)$
such that
\begin{equation*}
\begin{split}
\theta\in L^\infty(\R_+;H^1),~
\partial_1\theta\in L^2(\R_+&;H^1\cap L^\infty),~\omega\in L^\infty_{loc}(\R_+;L^\infty),\\
|\partial_1|^s\theta\in L^\infty(\R_+;L^2)&,~
|\partial_1|^{1+s}\theta\in L^2_{loc}(\R_+;L^2).
\end{split}
\end{equation*}
\end{theorem}
Then we give the proposition which showing the regularity estimate of $(\omega,\theta)$.
\begin{prop}
\label{prop-omega-theta}
Let $0<s<1$, assume the initial data $\omega_0\in L^2\cap L^\infty\cap H^s$ and $\theta_0\in H^{1+s}$. Then the following estimate holds true,
\begin{equation}
\begin{split}
\|\omega(t)\|_{H^s}^2+\|\theta(t)\|_{H^{1+s}}^2+\int_0^t\|\partial_1\theta(\tau)\|_{H^{1+s}}^2~d\tau\leq C(t)e^{\int_0^t\|\nabla u(\tau)\|_{L^\infty}~d\tau}.
\end{split}
\end{equation}
\end{prop}
\begin{proof}
We first estimate $\omega$.
Applying $\Delta_q$ to \eqref{vorticity2}, we get
\begin{equation}
\partial_t\Delta_q\omega+\Delta_q(u\cdot\nabla\omega)=\partial_1\Delta_q\theta.
\end{equation}
Taking $L^2$ inner product with $\Delta_q\omega$, one can deduce
\begin{equation}
\label{omega-Hs}
\begin{split}
\frac12\frac{d}{dt}\|\Delta_q\omega(t)\|_{L^2}^2&=-\int_{\R^2}\Delta_q(u\cdot\nabla\omega)\Delta_q\omega~dx
+\int_{\R^2}\partial_1\Delta_q\theta\Delta_q\omega~dx\\
&\triangleq N_1+N_2.
\end{split}
\end{equation}
For $N_1$, making use of Lemma \ref{lemA.1},
\begin{equation}
\label{N_1}
\begin{split}
N_1\leq Cb_q2^{-2qs}\|\nabla u\|_{L^\infty}\|\omega\|_{H^s}^2.
\end{split}
\end{equation}
Then we estimate $N_2$, by H\"older inequality and Young's inequality,
\begin{equation}
\label{N_2}
\begin{split}
N_2\leq &C\|\partial_1\Delta_q\theta\|_{L^2}\|\Delta_q\omega\|_{L^2}
\leq Cb_q2^{-2qs}\|\partial_1\theta\|_{H^s}
\|\omega\|_{H^s}
\leq Cb_q2^{-2qs}(\|\theta\|_{H^{1+s}}^2+\|\omega\|_{H^s}^2).
\end{split}
\end{equation}
Inserting the estimate \eqref{N_1} and \eqref{N_2} into \eqref{omega-Hs}, then multiplying both side by $2^{2qs}$ and summing up over $q\geq-1$, we obtain
\begin{equation}
\label{omega-Hs-final}
\begin{split}
\frac12\frac{d}{dt}\|\omega(t)\|_{H^s}^2\leq C(1+\|\nabla u\|_{L^\infty})\times(\|\omega\|_{H^s}^2+\|\theta\|_{H^{1+s}}^2).
\end{split}
\end{equation}
Then we estimate $\theta$. Applying $\Delta_q$ to the second equation of \eqref{system1}, we obtain
\begin{equation}
\label{localiezd-theta}
\partial_t\Delta_q\theta+\Delta_q(u\cdot\nabla\theta)-\partial_1^2\Delta_q\theta=0.
\end{equation}
Multiplying \eqref{localiezd-theta}
by $\Delta_q \theta$ and integrating over $\R^2$ with respect to $x$, after integration by part, one can deduce
\begin{equation}
\label{theta-H1+s}
\begin{split}
\frac12\frac{d}{dt}\|\Delta_q\theta(t)\|_{L^2}^2+\|\partial_1\Delta_q\theta\|_{L^2}^2
&=-\int_{\R^2}\Delta_q(u\cdot\nabla\theta)\Delta_q\theta~dx\\
&=-\sum_{|k-q|\leq2}\int_{\R^2}\Delta_q(S_{k-1}u\cdot\nabla\Delta_k\theta)\Delta_q\theta~dx\\
&\quad-\sum_{|k-q|\leq2}\int_{\R^2}\Delta_q(\Delta_{k}u\cdot\nabla S_{k-1}\theta)\Delta_q\theta~dx\\
&\quad-\sum_{k\geq q-1}\sum_{|k-l|\leq1}\int_{\R^2}\Delta_q(\Delta_{k}u\cdot\nabla \Delta_{k+l}\theta)\Delta_q\theta~dx\\
&\triangleq \Theta_1+\Theta_2+\Theta_3.
\end{split}
\end{equation}
For $\Theta_1$, along the same method as in the proof of Lemma
\ref{lemA.1} which showed in the Appendix, we can obtain
\begin{equation}
\label{Theta1}
\begin{split}
\Theta_1\leq Cb_q2^{-2q(1+s)}\|\nabla u\|_{L^\infty}\|\theta\|_{H^{1+s}}^2.
\end{split}
\end{equation}
For $\Theta_2$, we can write it explicitly,
\begin{equation}
\label{Theta2}
\begin{split}
\Theta_2&=
-\sum_{|k-q|\leq2}\int_{\R^2}\Delta_q(\Delta_{k}u\cdot\nabla S_{k-1}\theta)\Delta_q\theta~dx\\&=-\sum_{|k-q|\leq2}\int_{\R^2}\Delta_q(\Delta_{k}u^1\partial_1 S_{k-1}\theta)\Delta_q\theta~dx\\
&\quad-\sum_{|k-q|\leq2}\int_{\R^2}\Delta_q(\Delta_{k}u^2\partial_2 S_{k-1}\theta)\Delta_q\theta~dx\\
&\triangleq \Theta_{21}+\Theta_{22}.
\end{split}
\end{equation}
Making use of H\"older inequality, $\Theta_{21}$ can be bounded by
\begin{equation*}
\begin{split}
\Theta_{21}&\leq C \sum_{|k-q|\leq2} \|\Delta_{k}u^1\|_{L^2}\|\partial_1\theta\|_{L^\infty}\|\Delta_{q}\theta\|_{L^2}\\
&\leq C \sum_{|k-q|\leq2} \|\Delta_{k}\partial_2\Delta^{-1}\omega\|_{L^2}\|\partial_1\theta\|_{L^\infty}\|\Delta_{q}\theta\|_{L^2}\\
&\leq C \sum_{|k-q|\leq2}2^{-k}2^{-s k}2^{s k} \|\Delta_{k}\omega\|_{L^2}\|\partial_1\theta\|_{L^\infty}2^{-{(1+s) k}}2^{{(1+s) k}}\|\Delta_{q}\theta\|_{L^2}\\
&\leq C 2^{-2(1+s)q}b_q\|\partial_1\theta\|_{L^\infty} \|\omega\|_{H^s}\|\theta\|_{H^{1+s}},
\end{split}
\end{equation*}
where we have used the Biot-Savart law \eqref{biot-savart}.\\
Also making use of \eqref{biot-savart}, combining with
integration by part, we can write $\Theta_{22}$ as
\begin{equation*}
\begin{split}
\Theta_{22}&=-\sum_{|k-q|\leq2}\int_{\R^2}\Delta_q(\Delta_{k}u^2\partial_2 \Delta_{k+l}\theta)\Delta_q\theta~dx\\
&=-\sum_{|k-q|\leq2}\int_{\R^2}\Delta_q(\Delta_{k}\partial_1\Delta^{-1}\omega\partial_2 S_{k-1}\theta)\Delta_q\theta~dx\\
&=\sum_{|k-q|\leq2}\int_{\R^2}\Delta_q(\Delta_{k}\Delta^{-1}\omega\partial_2 S_{k-1}\theta)\partial_1\Delta_q\theta~dx\\
&\quad+\sum_{|k-q|\leq2}\int_{\R^2}\Delta_q(\Delta_{k}\Delta^{-1}\omega\partial_1\partial_2 S_{k-1}\theta)\Delta_q\theta~dx\\
&\triangleq \Theta_{221}+\Theta_{222}.
\end{split}
\end{equation*}
For $\Theta_{221}$, by H\"older inequality and Bernstein inequality in Lemma \ref{bernstein},
\begin{equation*}
\begin{split}
\Theta_{221}&\leq C \sum_{|k-q|\leq2}  \|\Delta_{k}\Delta^{-1}\omega\|_{L^\infty}\|\partial_2S_{k-1}\theta\|_{L^2}\|\partial_1\Delta_{q}\theta\|_{L^2}\\
&\leq C \sum_{|k-q|\leq2}2^{-2k} \|\Delta_{k}\omega\|_{L^\infty}\|\theta\|_{H^1}2^{-{(1+s) q}}2^{{(1+s) q}}\|\Delta_{q}\theta\|_{L^2}\\
&\leq C \sum_{|k-q|\leq2}2^{-2k} 2^{k}\|\Delta_{k}\omega\|_{L^2}\|\theta\|_{H^1}2^{-{(1+s) q}}2^{{(1+s) q}}\|\partial_1\Delta_{q}\theta\|_{L^2}\\
&\leq C 2^{-2(1+s)q}b_q\|\theta\|_{H^1} \|\omega\|_{H^s}\|\partial_1\theta\|_{H^{1+s}}.
\end{split}
\end{equation*}
Next we bound $\Theta_{222}$, by H\"older inequality and Bernstein inequality,
\begin{equation*}
\begin{split}
\Theta_{222}&\leq C \sum_{|k-q|\leq2}  \|\Delta_{k}\Delta^{-1}\omega\|_{L^2}\|\partial_1\partial_2S_{k-1}\theta\|_{L^\infty}\|\Delta_{q}\theta\|_{L^2}\\
&\leq C \sum_{|k-q|\leq2}2^{-2k} \|\Delta_{k}\omega\|_{L^2}\bigg(\sum_{k'\leq k-2}\|\partial_1\partial_2\Delta_{k'}\theta\|_{L^\infty}\bigg)\|\Delta_{q}\theta\|_{L^2}\\
&\leq C \sum_{|k-q|\leq2}2^{-2k} \|\Delta_{k}\omega\|_{L^2}2^{(1-s)q}\bigg(\sum_{k'\leq k-2}2^{(k'-k)(1-s)}2^{sk'}\|\partial_1\partial_2\Delta_{k'}\theta\|_{L^2}\bigg)\|\Delta_{q}\theta\|_{L^2}\\
&\leq C 2^{-2(1+s)q}b_q\|\omega\|_{L^2} \|\theta\|_{H^s}\|\partial_1\theta\|_{H^{1+s}},
\end{split}
\end{equation*}
where we have used the discrete
Young's inequality in the last step.\\
Then inserting the estimates of $\Theta_{21}$, $\Theta_{221}$ and $\Theta_{222}$ into \eqref{Theta2},
one can obtain
\begin{equation}
\label{Theta2-final}
\begin{split}
\Theta_{2}&\leq C  2^{-2(1+s)q}b_q\|\partial_1\theta\|_{L^\infty} \|\omega\|_{H^s}\|\theta\|_{H^{1+s}}\\
&\quad+ C 2^{-2(1+s)q}b_q(\|\omega\|_{H^s}+ \|\theta\|_{H^{1+s}})\|\partial_1\theta\|_{H^{1+s}}.\\
\end{split}
\end{equation}
Finally we estimate $\Theta_3$, by H\"older inequality and Bernstein inequality,
\begin{equation}
\label{Theta3}
\begin{split}
\Theta_{3}&=-\sum_{k\geq q-1}\sum_{|k-l|\leq1}
\int_{\R^2}\Delta_q\nabla\cdot(\Delta_ku\Delta_l\theta)\Delta_q\theta~dx\\
&\leq C \sum_{k\geq q-1}2^q\|\Delta_k u\|_{L^\infty}\|\Delta_k \theta\|_{L^2}\|\Delta_q \theta\|_{L^2}\\
&\leq C \sum_{k\geq q-1,~k\geq 0}2^{q-k}\|\Delta_k\nabla u\|_{L^\infty}\|\Delta_k \theta\|_{L^2}\|\Delta_q \theta\|_{L^2}+
C\|\Delta_{-1}u\|_{L^\infty}\|\Delta_{-1}\theta\|_{L^2}^2\\
&\leq C  2^{-2(1+s)q}b_q(1+\|\nabla u\|_{L^\infty}) \|\theta\|_{H^{1+s}}^2.
\end{split}
\end{equation}
Inserting the estimates \eqref{Theta1}, \eqref{Theta2-final} and \eqref{Theta3} into \eqref{theta-H1+s}, and making use of Young's inequality, we obtain
\begin{equation*}
\begin{split}
\frac12\frac{d}{dt}\|\Delta_q\theta(t)\|_{L^2}^2+\|\partial_1\Delta_q\theta\|_{L^2}^2
\leq &\quad C2^{-2(1+s)q}(1+\|\nabla u\|_{L^\infty}+\|\partial_1\theta\|_{L^\infty})\times(\|\omega\|_{H^s}^2+\|\theta\|_{H^{1+s}}^2)\\
&+\frac{\varepsilon}{\|b_q\|_{\ell^1}}b_q2^{-2(1+s)q}\|\partial_1\theta\|_{H^{1+s}}.
\end{split}
\end{equation*}
Multiplying both sides by $2^{(1+s)q}$ and summing up from $-1$ to $\infty$ with respect to $q$, choosing $\varepsilon=\frac12$, one can deduce
\begin{equation}
\label{theta-H1+s-final}
\begin{split}
\frac{d}{dt}\|\theta(t)\|_{H^{1+s}}^2+\|\partial_1\theta\|_{H^{1+s}}^2
\leq  C(1+\|\nabla u\|_{L^\infty}+\|\partial_1\theta\|_{L^\infty})\times(\|\omega\|_{H^s}^2+\|\theta\|_{H^{1+s}}^2).
\end{split}
\end{equation}
Combining \eqref{omega-Hs-final} with \eqref{theta-H1+s-final} and by Gr\"onwall's Lemma, because $\partial_1\theta\in L^2_{t}(L^\infty_{x})$~(see Theorem \ref{result-DP-model1}), we obtain
\begin{equation*}
\begin{split}
\|\omega(t)\|_{H^s}^2+\|\theta(t)\|_{H^{1+s}}^2+\int_0^t|\partial_1\theta(\tau)\|_{H^{1+s}}^2~d\tau\leq C(t)e^{\int_0^t\|\nabla u(\tau)\|_{L^\infty}~d\tau},
\end{split}
\end{equation*}
which complete the proof of this proposition.
\end{proof}

\begin{subsection}{A priori estimates for the striated regularity}
~\\
In this subsection, we will give the estimates of
tangential derivatives of $\omega$ and the regularity estimates of $X$. The
first lemma gives $L^p$ $(p\in[1,\infty])$ estimate
of $X$.
\begin{lemma}
	\label{L-infty-X}
	Let $r\in[1,\infty]$, $X_0\in L^r$ and $(\omega_0,\theta_0)$ satisfies the assumption in Lemma \ref{lipschitz-information}. Then the solution $X$ of equation \eqref{X-eq} satisfies
	\begin{equation}
	\label{X-Lr-estimate}
	\|X_0\|_{L^r}e^{-\int_0^t\|\nabla u(\tau)\|_{L^\infty}d\tau}\leq\|X(t)\|_{L^r}\leq
	\|X_0\|_{L^r}e^{\int_0^t\|\nabla u(\tau)\|_{L^\infty}d\tau}.
	\end{equation}
\end{lemma}
\begin{proof}
Multiplying both side of equation \eqref{X-eq} by  $|X|^{r-2}X~(1<r<\infty)$ and integrating over $\R^2$
with respect to $x$, we can obtain
\begin{equation}
\frac1r\frac{d}{dt}\|X(t)\|_{L^r}^r\leq C\|\nabla u\|_{L^\infty}\|X\|_{L^r}^{r},
\end{equation}
which implies the right hand side inequality of \eqref{X-Lr-estimate}.
Using the time reversibility of this equation and the same $L^r$ estimate, we can obtain the first inequality of 
\eqref{X-Lr-estimate}. Then taking $r\rightarrow \infty$, we can deduct the result for the case $r=\infty$, which complete the proof of this lemma.

\end{proof}

 Applying $\partial_X$ to the vorticity equation, according to \eqref{properties-partialX}, we get
$\partial_X\omega$ satisfies the following equation
\begin{equation}
\label{paX-omega-eq}
\partial_t\partial_X\omega+u\cdot\nabla\partial_X\omega=\partial_X(\partial_1\theta)=X\cdot\nabla\partial_1\theta.
\end{equation}
The next lemma deals with the $L^p$ estimate of $\partial_X\omega$.
\begin{lemma}
\label{paX-omega}
Let $\partial_{X_0}\omega_0$ $\in L^p~(2\leq p<\infty)$, and $(\omega,\theta)$ satisfies the assumptions in Proposition \ref{prop-omega-theta}, then we have
\begin{equation*}
\begin{split}
\|\partial_X\omega(t)\|_{L^p}\leq \|\partial_{X_0}\omega_0\|_{L^p}+
C(t)e^{2\int_0^t\|\nabla u(\tau)\|_{L^\infty}~d\tau}.
\end{split}
\end{equation*}
\end{lemma}
\begin{proof}
Multiplying the equation \eqref{paX-omega-eq} by $|\partial_X\omega|^{p-2}\partial_X\omega$~$(2\leq p<\infty)$, and integrating over $\R^2$ with respect to $x$, because $u$ satisfies the divergence-free condition, by H\"older inequality,
\begin{equation*}
\frac{1}{p}\frac{d}{dt}\|\partial_X\omega(t)\|_{L^p}^p\leq \|X\|_{L^\infty}\|\partial_1\nabla\theta\|_{L^p}\|\partial_X\omega\|_{L^p}^{p-1}.
\end{equation*}
Because of the embedding $H^s\hookrightarrow L^p$ with $\frac{2}{p}=1-s$, we obtain
\begin{equation*}
\frac{d}{dt}\|\partial_X\omega(t)\|_{L^p}\leq \|X\|_{L^\infty}\|\partial_1\nabla\theta\|_{H^s}.
\end{equation*}
Then integrating in time and combining with the result of Proposition \ref{prop-omega-theta},
\begin{equation*}
\begin{split}
\|\partial_X\omega(t)\|_{L^p}&\leq \|\partial_{X_0}\omega_0\|_{L^p}+\int_0^t \|X(\tau)\|_{L^\infty}\|\partial_1\nabla\theta(\tau)\|_{H^s}~d\tau\\
&\leq \|\partial_{X_0}\omega_0\|_{L^p}+ \|X\|_{L^\infty_{t,x}}\int_0^t\|\partial_1\nabla\theta(\tau)\|_{H^s}~d\tau\\
&\leq \|\partial_{X_0}\omega_0\|_{L^p}+
C(t)e^{2\int_0^t\|\nabla u(\tau)\|_{L^\infty}~d\tau},
\end{split}
\end{equation*}
which complete the proof of this lemma.
\end{proof}

Then we give the H\"older estimate for $X$. The next proposition obtain the Lipschitz information of the velocity $u$ and the $C^s$
norm of $X$ simultaneously.
\begin{prop}
Let $0<s<1$, assume $X_0\in C^s$, $\partial_{X_0}\omega_0\in L^p$ and $(\omega_0,\theta_0)$
satisfies the assumptions in Proposition \ref{prop-omega-theta}, the we have the velocity u satisfies
\begin{equation}
\nabla u\in L^1([0,t];L^\infty).
\end{equation}
Moreover,
\begin{equation}
\begin{split}
X\in L^\infty([0,t];C^s),~\omega\in L^\infty([0,t];&H^s),~\partial_X\omega\in L^\infty([0,t];L^p).\\
\theta\in L^\infty([0,t];H^{1+s})&, ~
\partial_1\theta\in L^2([0,t];H^{1+s}).
\end{split}
\end{equation}

\end{prop}
\begin{proof}
Firstly, we compute the H\"older estimate of $X$. Applying Lemma \ref{holder-transport} to \eqref{X-eq}, we obtain
\begin{equation}
\label{X-Cs}
\begin{split}
\|X(t)\|_{C^s}&\leq C\|X_0\|_{C^s}e^{\tilde{C}\int_0^t\|\nabla u(\tau)\|_{L^\infty}~d\tau}+C\int_0^t\|\partial_X u(\tau)\|_{C^s}e^{\tilde{C}\int_{\tau}^t\|\nabla u(s)\|_{L^\infty}~ds}~d\tau\\
&\leq Ce^{\tilde{C}\int_0^t\|\nabla u(\tau)\|_{L^\infty}~d\tau}(\|X_0\|_{C^s}+\int_0^t\|\partial_X u(\tau)\|_{C^s}e^{-\tilde{C}\int_0^{\tau}\|\nabla u(s)\|_{L^\infty}~ds}~d\tau),
\end{split}
\end{equation}
where we can choose $\tilde{C}>2$.
In order to estimate H\"older norm of $\partial_Xu$, we need the following estimate which proof can be found in \cite{Chemin-perfect,BCD},
\begin{equation}
\label{paXu}
\begin{split}
\|\partial_Xu\|_{C^s}\leq C(\|\nabla u\|_{L^\infty}\|X\|_{C^s}+\|\partial_X\omega\|_{C^{s-1}}).
\end{split}
\end{equation}
By Sobolev embedding $L^p\hookrightarrow C^{s-1}$~$(1-s=\frac2p)$ and Lemma \ref{paX-omega}, we obtain
\begin{equation}
\label{paX-omega-Cs-1}
\begin{split}
\|\partial_X\omega\|_{C^{s-1}}\leq C \|\partial_X\omega\|_{L^p}\leq C\|\partial_X\omega_0\|_{L^p}+C(t)e^{2\int_{0}^{t}\|\nabla u(\tau)\|_{L^\infty}~d\tau}.
\end{split}
\end{equation}
Inserting \eqref{paXu} and \eqref{paX-omega-Cs-1} into \eqref{X-Cs}, one can deduce that
\begin{equation*}
\begin{split}
\|X(t)\|_{C^s}&\leq Ce^{\tilde{C}\int_0^t\|\nabla u(\tau)\|_{L^\infty}~d\tau}\bigg(\|X_0\|_{C^s}+\int_0^t(C(\tau)\\&\quad\quad+\|\nabla u(\tau)\|_{L^\infty}\|X(\tau)\|_{C^s}e^{-\tilde{C}\int_0^{\tau}\|\nabla u(s)\|_{L^\infty}~ds})~d\tau\bigg).
\end{split}
\end{equation*}
Denoting
\begin{equation*}
\begin{split}
F(t)\triangleq \|X(t)\|_{C^s}e^{\tilde{C}\int_0^t\|\nabla u(\tau)\|_{L^\infty}~d\tau}.
\end{split}
\end{equation*}
Then according to the above estimates, we obtain
\begin{equation*}
\begin{split}
F(t)\leq CF(0)+\int_0^tC(\tau)(\|\nabla u(\tau)\|_{L^\infty}+1)(F(\tau)+1)~d\tau.
\end{split}
\end{equation*}
By Gr\"onwall's Lemma,
\begin{equation*}
\begin{split}
F(t)\leq C(F(0)+1)e^{\int_0^tC(\tau)(\|\nabla u(\tau)\|_{L^\infty}+1)~d\tau}.
\end{split}
\end{equation*}
According to the definition of $F(t)$, we obtain the H\"older estimate of $X$ that,
\begin{equation}
\label{X-Cs-final}
\begin{split}
\|X(t)\|_{C^s}\leq C(t)e^{C\int_0^t\|\nabla u(\tau)\|_{L^\infty}~d\tau}.
\end{split}
\end{equation}
Recalling the logarithmic inequality in Lemma \ref{log-ineq-lemma} that
\begin{equation}
\label{log-ineq-omega}
\|\nabla v\|_{L^\infty}\leq
C\bigg(\|\omega\|_{L^2}+\|\omega\|_{L^\infty}\log\bigg(e+\frac{\|\omega\|_{C^s_X}}{\|\omega\|_{L^\infty}}\bigg)\bigg),
\end{equation}
where $\|\omega\|_{C^s_X}$ is defined in Definition \ref{definition2.1}.\\
Because $\|\omega\|_{L^2\cap L^\infty}$ is bounded, inserting the estimates \eqref{paX-omega-Cs-1}, \eqref{X-Cs-final}
into \eqref{log-ineq-omega}, we obtain
\begin{equation*}
\begin{split}
\|\nabla v\|_{L^\infty}&\leq
C\bigg(1+\log\bigg(e+C(t)e^{C\int_0^t\|\nabla u(\tau)\|_{L^\infty}}\bigg)\bigg)\\
&\leq
C\bigg(1+\int_0^tC(t)(1+\|\nabla u(\tau)\|_{L^\infty})~d\tau\bigg).
\end{split}
\end{equation*}
Then by Gr\"onwall's Lemma,
\begin{equation}
\label{lipschitz-norm-u}
\|\nabla u(t)\|_{L^\infty}\leq C(t),~~~~\forall t>0.
\end{equation}
Combining the estimates \eqref{X-Cs-final} and \eqref{lipschitz-norm-u}, we can obtain the
desired H\"older norm of $X$, Then inserting the estimate \eqref{X-Cs-final} into Proposition
\ref{prop-omega-theta} and Lemma \ref{paX-omega}, we can complete the proof of this proposition.
\end{proof}
\subsection{The vortex patch problem}
In this subsection, we devote to prove Corollary \ref{cor-1}, which solving the vortex patch problem. Because
\begin{equation*}
\omega_0=\chi_{D_0}(x)\triangleq\left\{\begin{split}
1\quad x\in D_0,\\
0\quad x\notin D_0,
\end{split}\right.
\end{equation*}
where $D_0$ is a connected bounded domain with $\pa D_0\in C^{1+s}$ for $0<s<1$. Then according to Definition 
\ref{Cs-class}, there exist a real function $f_0\in C^{1+s}$ and a neighborhood $V_0$ such that $\pa D_0=V_0\cap f^{-1}({0})$ and $\nabla f_0\neq 0$ on $V_0$. Noticing that at time $t$, the boundary $\pa D_t=\psi(D_0,t)$ is the level set of the function $f(\cdot,t)=f_0(\psi^{-1}(\cdot,t))$ with $f$ being transported by the flow $\psi$:
\begin{equation}
\label{f-eq}
\left \{\begin{split} &\pa_t f+u\cdot \nabla f=0,\\
&f(x,0)=f_0(x).
\end{split}
\right.
\end{equation} 

Setting the vector field  $X\triangleq\nabla^{\perp}f$ with initial data  $X_0\triangleq\nabla^{\perp}f_0$, 
it is not hard to verifies that $X$ satisfying \eqref{X-defin} and the corresponding system $\eqref{X-eq}$. 
Then we can parametrize $\pa D_0$ as
\begin{equation*}
\gamma_0:\mathbb{S}^1\rightarrow\pa D_0,~~\text{via}~~\sigma\mapsto\gamma_0(\sigma),
\end{equation*}
with
\begin{equation}
\label{gamma0}
\left \{\begin{split} &\pa_{\sigma} \gamma_0=X_0(\gamma_0(\sigma)),~~\forall~\sigma\in\mathbb{S}^1,\\
&\gamma_0(0)=x_0\in\pa D_0.
\end{split}
\right.
\end{equation} 
In order to conclude the proof of Corollary \ref{cor-1}, we observe that a parametrization for $\pa D_t$ is given by $\gamma_t(\sigma)\triangleq \psi(\gamma_0(\sigma),t)$ and by differentiating with respect to the parameter $\sigma$, we get
\begin{equation}
\label{gamma-t}
\left \{\begin{split} &\pa_{\sigma} \gamma_t(\sigma)=X(\gamma_t(\sigma)),~~\forall~\sigma\in\mathbb{S}^1,\\
&\gamma_t(0)=\psi(x_0,t)\in\pa D_t.
\end{split}
\right.
\end{equation} 
According to Theorem \ref{theorem-model1}, $X\in L^\infty([0,T];C^s)$, thus $\gamma_t\in C^{1+s}(\mathbb{S}^1)$ for all $t\geq0$.
This completes the proof of Corollary \ref{cor-1}.

\end{subsection}

\end{section}

\begin{section}{The Case of Horizontal Viscosity}
In this section, we focus on system \eqref{system2}. Before we begin to prove the result in Theorem \ref{theorem-model2}, we need to review the following existence and uniqueness result for system \eqref{system2} which can be found in \cite{DP1}.
\begin{theorem}
	\label{result-DP-model2}
Let $s\in(\frac12,1]$. For all function $\theta_0\in H^s\cap L^\infty$ and divergence-free vector field $u_0\in H^1$ with vorticity $\omega_0\in\sqrt L$. System \eqref{system2} with data $(u_0,\theta_0)$ admits a unique global solution $(u,\theta)$ such that $\theta\in C_{w}(\R_+;L^\infty)\cap C(\R_+;H^{s-\varepsilon})$ for all $\varepsilon>0$ and
\begin{equation}
u\in C_{w}(\R_+;H^1),~~~~\omega\in L^\infty_{loc}(\R_+;\sqrt L)~~and~~\nabla u \in L^\infty_{loc}(\R_+;\sqrt L).
\end{equation}
\end{theorem}
In the rest of this section, we will first show that the solution $u$ of system \eqref{system2} actually can be in $L^1([0,t]; L^\infty)$ in the first subsection. Then we estimate the straited regularity in the second subsection. In subsections 4.3-4.4, we exam the higher regularity estimate of $(\omega, \theta)$ and the vector field $X$. The proof of Corollary \ref{cor-2} will be given in the last subsection.
	
\subsection{A priori estimates for the Lipschitz norm of the velocity field}

In this subsection, we will give the estimates for the Lipschitz norm of the velocity field and $H^s~(\frac12<s<1)$ norm of $(\omega,\theta)$. Those estimate will be based on the following global existence theorem \cite{DP1}.

Then we give the estimate for $\|\nabla u\|_{L^\infty}$, which plays and important role in the estimate for striated regularity in the next subsections. The main results can be stated as follows.
\begin{lemma}
\label{lipschitz-information}
Assume $\omega_0\in H^s$ and $\theta_0\in H^\beta$ with $\beta>s>\frac12$, then the solution $(\omega,\theta)$ satisfies
\begin{equation*}
\begin{split}
\|\omega\|_{L^\infty_t(H^s)}^2+\|\partial_1\nabla\omega\|_{L^2_t(H^s)}^2\leq C,\quad\quad\|\theta\|_{L^\infty_t(H^s)}^2\leq C,
\end{split}
\end{equation*}
moreover,
\begin{equation*}
\begin{split}
\|\nabla u\|_{L^2_t(L^{\infty})}\leq C.
\end{split}
\end{equation*}
\end{lemma}
\begin{proof}
Because of Theorem \ref{result-DP-model2}, we have already know $\nabla u\in\sqrt{L}$.
Then according to the definition of space $\sqrt{L}$ and Lemma \ref{lemma-losing-transport}, we know
\begin{equation}
\label{theta-Hs-estimate}
\|\theta\|_{L^\infty([0,t];H^s)}\leq C(t).
\end{equation}
Then we give the estimate of $\omega$. Applying $\Delta_q$ to the vorticity equation \eqref{vorticity1} and taking $L^2$ inner product with $\Delta_q\omega$, one can obtain
\begin{equation}
\frac12\frac {d}{dt}\|\Delta_q\omega(t)\|_{L^2}^2+\|\partial_1\Delta_q\omega\|_{L^2}^2
= \int_{\R^2}\partial_1\Delta_q\theta\Delta_q\omega~dx-\int_{\R^2}\Delta_q(u\cdot\nabla\omega)\Delta_q\omega~dx.
\end{equation}
After integration by part, according to H\"older inequality and Young's inequality,
\begin{equation}
\label{eq-omega}
\frac12\frac {d}{dt}\|\Delta_q\omega(t)\|_{L^2}^2+\frac12\|\partial_1\Delta_q\omega\|_{L^2}^2
\leq C\|\Delta_q\theta\|_{L^2}^2-\int_{\R^2}\Delta_q(u\cdot\nabla\omega)\Delta_q\omega~dx.
\end{equation}
By Lemma \ref{lemA.1} and Young's inequality,
\begin{equation}
\label{equation1}
\begin{split}
-\int_{\R^2}\Delta_q(u\cdot \nabla \omega)\Delta_q \omega~dx\leq&  C b_q2^{-2qs}(\|u\|_{L^2}+\|\omega\|_{L^2}+\|\partial_1\omega\|_{L^2})\\&\quad\quad\times(\|\omega\|_{H^s}^2+\|\omega\|_{H^s}^\frac12\|\partial_1\omega\|_{H^s}^\frac12+\|\omega\|_{H^s}^\frac32\|\partial_1\omega\|_{H^s}^\frac12)\\
\leq& C b_q2^{-2qs} \|\omega\|_{H^s}^2+\frac{1}{4\|b_q\|_{\ell^1}}b_q2^{-2qs} \|\partial_1\omega\|_{H^s}^2.
\end{split}
\end{equation}
According to the bound \eqref{theta-Hs-estimate},
\begin{equation}
\label{equation2}
\begin{split}
\|\Delta_q\theta\|_{L^2}^2\leq b_q2^{-2qs}\|\theta\|_{H^s}^2.
\end{split}
\end{equation}
Inserting \eqref{equation1}, \eqref{equation2} into \eqref{eq-omega} and taking summation of $q$, after calculation we obtain
\begin{equation*}
\frac {d}{dt}\|\omega(t)\|_{H^s}^2+\|\partial_1\omega\|_{H^s}^2
\leq C(1+\|\omega(t)\|_{H^s}^2).
\end{equation*}
Then by Gr\"onwall's Lemma, we get
\begin{equation*}
\|\omega(t)\|_{H^s}^2+\int_{0}^{t}\|\partial_1\omega(\tau)\|_{H^s}^2~d\tau
\leq C.
\end{equation*}
According to trace theory, we know
\begin{equation*}
\|f(x_1,x_2)\|_{L^\infty_{x_2}(H^{\alpha-\frac12}_{x_1})}\leq C\|f(x_1,x_2)\|_{H^{\alpha}},\quad\text{for}~\alpha>\frac12.
\end{equation*}
Thus by Sobolev embedding,
\begin{equation*}
\begin{split}
\int_{0}^{t}\|\omega(\tau)\|_{L^\infty}^2~d\tau&\leq
\int_{0}^{t}\|\omega(\tau)\|_{L^\infty_{x_2}(H^{s+\frac12}_{x_1})}^2~d\tau\\
&\leq \int_{0}^{t}\|\partial_1\omega(\tau)\|_{L^\infty_{x_2}(H^{s-\frac12}_{x_1})}^2~d\tau\\
&\leq \int_{0}^{t}\|\partial_1\omega(\tau)\|_{H^{s}}^2~d\tau\\
&\leq C.
\end{split}
\end{equation*}
\noindent
Noticing that $\partial_1\omega=\Delta u^2$ and $\partial_1u^1+\partial_2u^2=0$, we have
\begin{equation*}
\int_{0}^{t}\|\partial_iu^j(\tau)\|_{H^{s+1}}^2~d\tau\leq C,\quad \text{for}~i,j=1,2,~(i,j)\neq(2,1).
\end{equation*}
Then by Sobolev embedding,
\begin{equation*}
\int_{0}^{t}\|\partial_iu^j(\tau)\|_{L^\infty}^2~d\tau\leq C,\quad \text{for}~i,j=1,2,~(i,j)\neq(2,1).
\end{equation*}
As for $(i,j)=(2,1)$, according to the definition of vorticity $\omega$,
\begin{equation*}
\partial_2u^1=\partial_1u^2-\omega,
\end{equation*}
so
\begin{equation*}
\int_{0}^{t}\|\partial_2u^1(\tau)\|_{L^\infty}^2~d\tau\leq \int_{0}^{t}\|\partial_1u^2(\tau)\|_{L^\infty}^2~d\tau+
\int_{0}^{t}\|\omega(\tau)\|_{L^\infty}^2~d\tau\leq C.
\end{equation*}
Thus we obtain $\|\nabla u\|_{L^2_t(L^{\infty})}$ is bounded, which completes the proof of this lemma.
\end{proof}
\subsection{A priori estimates for striated regularity.}
~\\
In this section, we will give some estimates about the vector field $X$. Along the same method of Lemma \ref{X-Lr-estimate} and combining with Lemma \ref{lipschitz-information},
one can deduce for any $r\in [1,\infty]$,
\begin{equation}
\|X(t)\|_{L^r}\leq C\|X_0\|_{L^r}e^{\int_0^t\|\nabla u(\tau)\|_{L^\infty}}~d\tau\leq C(t).
\end{equation}

The next lemma shows the $H^s~(\frac12<s<1)$ estimate for $X$. 

\begin{lemma}
\label{Hs-X}
Let $s>\frac{1}{2}$, $X_0\in H^s$ and $(\omega_0,\theta_0)\in H^s\times H^\beta$ with $\beta>s$. Then the solution $X$ of \eqref{X-eq} satisfies
\begin{equation*}
X\in L^\infty([0,t];H^s),
\end{equation*}
for any $t>0$.
\end{lemma}
\begin{proof}
Applying operator $\Delta_q$ to \eqref{X-eq},
\begin{equation}
\partial_t\Delta_qX+\Delta_q(u\cdot\nabla X)=\Delta_q\partial_Xu.
\end{equation}
Taking the $L^2$ inner product of the above equality with $\Delta_q X$, we get
\begin{equation}
\label{Hs-X-estimate}
\begin{split}
\frac{1}{2}\frac{d}{dt}\|\Delta_qX(t)\|_{L^2}^2&=-\int_{\R^2}\Delta_q(u\cdot\nabla X)\cdot \Delta_q X~d\tau+\int_{\R^2}\Delta_q\partial_Xu\cdot \Delta_q X~d\tau.\\
\end{split}
\end{equation}
For the first term of the right hand side in \eqref{Hs-X-estimate}. By Lemma \ref{lemA.1}, we have
\begin{equation}
\begin{split}
\label{X-Hs-term1}
-\int_{\R^2}\Delta_q(u\cdot \nabla X)\Delta_q X~dx\leq&  C b_q2^{-2qs}\|\nabla u\|_{L^\infty}
\|X\|_{H^s}^2.
\end{split}
\end{equation}
Then we estimate the last term of \eqref{Hs-X-estimate}, by H\"older inequality,
\begin{equation*}
\begin{split}
\int_{\R^2}\Delta_q\partial_Xu\cdot \Delta_q X~d\tau\leq&  \|\Delta_q\partial_Xu\|_{L^2}
\|\Delta_q X\|_{L^2}\leq Cb_q2^{-2qs}\|\partial_X u\|_{H^s}
\|X\|_{H^s}.
\end{split}
\end{equation*}
For $H^s$ norm of $\partial_Xu$, we can bound it by
\begin{equation*}
\begin{split}
\|\partial_X u\|_{H^s}=\|X\cdot\nabla u\|_{H^s}\leq C
(\|X\|_{L^\infty}\|\nabla u\|_{H^s}+\|X\|_{H^s}\|\nabla u\|_{L^\infty}).
\end{split}
\end{equation*}
By Lemma \ref{lipschitz-information} and Lemma \ref{L-infty-X}, we have known
\begin{equation*}
\begin{split}
\|X\|_{L^\infty}\leq C,~~\quad\quad\quad~~\|\nabla u\|_{H^s}\leq C\|\omega\|_{H^s}\leq C.
\end{split}
\end{equation*}
Thus we obtain
\begin{equation}
\label{X-Hs-term2}
\begin{split}
\int_0^t\Delta_q\partial_Xu\cdot \Delta_q X~d\tau\leq Cb_q2^{-2qs}
(\|X\|_{H^s}+\|\nabla u\|_{L^\infty}\|X\|_{H^s}^2).
\end{split}
\end{equation}
Inserting the estimates \eqref{X-Hs-term1} and \eqref{X-Hs-term2} into \eqref{Hs-X-estimate} then multiplying both sides by $2^{2qs}$ and taking summation over $q\geq-1$, we obtain
\begin{equation}
\label{Hs-X-estimate-final}
\begin{split}
\frac{1}{2}\frac{d}{dt}\|X(t)\|_{H^s}^2&\leq C(\|X\|_{H^s}+\|\nabla u\|_{L^\infty}\|X\|_{H^s}^2).
\end{split}
\end{equation}
Then by Gr\"onwall's Lemma and combining with Lemma \ref{lipschitz-information}, we
obtain
\begin{equation*}
\begin{split}
\|X\|_{H^s}\leq C(t),
\end{split}
\end{equation*}
which completes the proof of this lemma.

\end{proof}

\subsection{A priori estimates for $\omega$ and $\theta$.}
~\\
In this subsection, we will give some  regularity estimates for $(\omega,\theta)$ based on the Lipschitz information $\|\nabla u\|_{L^1_t(L^\infty_x)}$.
The following lemma gives the $H^1$ estimate
of $(\omega,\theta)$.
\begin{lemma}
Assume $\omega_0\in H^1$ and $\theta_0\in H^1$, then the solution $(\omega,\theta)$ satisfies
\begin{equation}
\begin{split}
\|\nabla\omega\|_{L^\infty_t(L^2)}^2+\|\nabla\theta\|_{L^\infty_t(L^2)}^2+\|\partial_1\nabla\omega\|_{L^2_t(L^2)}^2\leq C.
\end{split}
\end{equation}
\end{lemma}
\begin{proof}
Applying $\partial_k~(k=1,2)$ to the vorticity equation \eqref{vorticity1}, we can obtain $\partial_k \omega$ satisfies
\begin{equation}
\label{nabla-omega}
\begin{split}
\partial_t \partial_k\omega+u\cdot\nabla \partial_k\omega+ \partial_ku\cdot\nabla \omega-\partial_1^2\partial_k\omega=\partial_1\partial_k\theta.
\end{split}
\end{equation}
Multiplying $\partial_k \omega$ to \eqref{nabla-omega} and integrating over $\R^2$ with respect to $x$, we have
\begin{equation}
\label{H1--omega}
\begin{split}
\frac12\frac{d}{dt}\|\partial_k\omega(t)\|_{L^2}^2+\|\partial_1\partial_k\omega\|_{L^2}^2
&=\int_{\R^2}\partial_1\partial_k\theta \partial_k\omega~dx
-\int_{\R^2}\partial_ku\cdot \nabla \omega \partial_k\omega~dx\\
&\triangleq N_1+N_2.
\end{split}
\end{equation}
After integration by part and using H\"older inequality and Young's inequality, one can deduce
\begin{equation}
\label{N1}
\begin{split}
N_1\leq\frac12\|\partial_1\nabla\omega\|_{L^2}^2+\frac12\|\nabla\theta\|_{L^2}^2.
\end{split}
\end{equation}
For $N_2$, by H\"older inequality,
\begin{equation}
\label{N2}
\begin{split}
N_2\leq\|\nabla u\|_{L^\infty}\|\nabla \omega\|_{L^2}^2.
\end{split}
\end{equation}
Applying $\partial_k~(k=1,2)$ to the temperature equation of \eqref{system2}, we can obtain $\partial_k \theta$ satisfies
\begin{equation}
\label{nabla-theta}
\begin{split}
\partial_t \partial_k\theta+u\cdot\nabla \partial_k\theta+ \partial_ku\cdot\nabla \theta=0.
\end{split}
\end{equation}
Similarly, we can prove
\begin{equation}
\label{H1--theta}
\begin{split}
\frac12\frac{d}{dt}\|\partial_k\theta(t)\|_{L^2}^2\leq C\|\nabla u\|_{L^\infty}\|\nabla \theta\|_{L^2}^2.
\end{split}
\end{equation}
Inserting \eqref{N1} and \eqref{N2} into \eqref{H1--omega} and combining with \eqref{H1--theta}, we can deduce
\begin{equation*}
\begin{split}
\frac{d}{dt}(\|\nabla\omega(t)\|_{L^2}^2+\|\nabla\theta(t)\|_{L^2}^2)+\|\partial_1\nabla\omega(t)\|_{L^2}^2\leq C\|\nabla u\|_{L^\infty}(\|\nabla \omega\|_{L^2}^2+\|\nabla \theta\|_{L^2}^2).
\end{split}
\end{equation*}
Then by virtue of the Gr\"onwall's Lemma and Lemma \ref{lipschitz-information},
\begin{equation*}
\begin{split}
\|\nabla\omega\|_{L^\infty_t(L^2)}^2+\|\nabla\theta\|_{L^\infty_t(L^2)}^2+\|\partial_1\nabla\omega\|_{L^2_t(L^2)}^2\leq C,
\end{split}
\end{equation*}
which complete the proof of this lemma.
\end{proof}

The follow lemma shows the $L^p$ estimate
of $(\nabla\omega,\nabla\theta)$.
\begin{lemma}
\label{w1p-omega-theta}
Assume $\nabla\omega_0\in L^p$ and $\nabla\theta_0\in L^p$~$(2<p<\infty)$, then the solution $(\omega,\theta)$ satisfies
\begin{equation}
\begin{split}
\|\nabla\omega\|_{L^\infty_t(L^p)}^2+\|\nabla\theta\|_{L^\infty_t(L^p)}^2\leq C.
\end{split}
\end{equation}
\end{lemma}
\begin{proof}
Multiplying $|\partial_k \omega|^{p-2}\partial_k \omega$ to \eqref{nabla-omega} and integrating over $\R^2$ with respect to $x$, by H\"older inequality and Young's inequality,
\begin{equation*}
\begin{split}
&\quad\frac1p\frac{d}{dt}\|\partial_k\omega(t)\|_{L^p}^p+(p-1)\int_{\R^2}|\partial_1\partial_k\omega|^2|\partial_k\omega|^{p-2}~dx\\
&=\int_{\R^2}\partial_1\partial_k\theta |\partial_k\omega|^{p-2}\partial_k\omega~dx
-\int_{\R^2}\partial_ku\cdot \nabla \omega |\partial_k\omega|^{p-2}\partial_k\omega~dx\\
&=-(p-1)\int_{\R^2}\partial_k\theta |\partial_k\omega|^{p-2}\partial_1\partial_k\omega~dx
-\int_{\R^2}\partial_ku\cdot \nabla \omega |\partial_k\omega|^{p-2}\partial_k\omega~dx\\
&\leq \frac{p-1}{2}\int_{\R^2}|\partial_1\partial_k\omega|^2|\partial_k\omega|^{p-2}~dx+C\int_{\R^2}|\partial_k\theta|^2 |\partial_k\omega|^{p-2}~dx\\
&\quad+\|\nabla u \|_{L^\infty}\|\nabla \omega\|_{L^p}\\
&\leq \frac{p-1}{2}\int_{\R^2}|\partial_1\partial_k\omega|^2|\partial_k\omega|^{p-2}~dx+C\|\nabla \theta\|_{L^p}^2\|\nabla \omega\|_{L^p}^{p-2}
+\|\nabla u \|_{L^\infty}\|\nabla \omega\|_{L^p}^p.
\end{split}
\end{equation*}
Thus we obtain
\begin{equation}
\label{Lp--nabla-omega}
\begin{split}
\frac{d}{dt}\|\partial_k\omega(t)\|_{L^p}^2\leq C\|\nabla \theta\|_{L^p}^2
+\|\nabla u \|_{L^\infty}\|\nabla \omega\|_{L^p}^2.
\end{split}
\end{equation}
Similarly, we can prove
\begin{equation}
\label{Lp--nabla-theta}
\begin{split}
\frac{d}{dt}\|\partial_k\theta(t)\|_{L^p}^2\leq C\|\nabla u\|_{L^\infty}\|\nabla \theta\|_{L^p}^2.
\end{split}
\end{equation}
Combining \eqref{Lp--nabla-omega} with \eqref{Lp--nabla-theta}, we can deduce
\begin{equation*}
\begin{split}
\frac{d}{dt}(\|\nabla\omega(t)\|_{L^p}^2+\|\nabla\theta(t)\|_{L^p}^2)\leq C(1+\|\nabla u\|_{L^\infty})(\|\nabla \omega\|_{L^p}^2+\|\nabla \theta\|_{L^p}^2).
\end{split}
\end{equation*}
Then by virtue of the Gr\"onwall's Lemma and Lemma \ref{lipschitz-information},
\begin{equation*}
\begin{split}
\|\nabla\omega\|_{L^p}^2+\|\nabla\theta\|_{L^p}^2\leq C(t),
\end{split}
\end{equation*}
which completes the proof of this lemma.
\end{proof}

Next we discuss the higher order regularity estimate for $(\omega,\theta)$.
Applying $|D|^s~(s>0)$ to the vorticity equation \eqref{vorticity1} and temperature equation of \eqref{system2}, we can get $(|D|^s\omega, |D|^s \theta)$ satisfies the following system,
\begin{equation}
\label{Hs-system}
\left\{
\begin{array}{cc}
\begin{split}
&\partial_t |D|^s\omega+u\cdot\nabla |D|^s\omega-\partial_1^2|D|^s\omega=\partial_1|D|^s\theta-[|D|^s,u\cdot\nabla]\omega,\\
&\partial_t |D|^s\theta+u\cdot\nabla|D|^s\theta=-[|D|^s,u\cdot\nabla]\theta.
\end{split}
\end{array}
\right.
\end{equation}
The follow lemma gives the $H^s~(s>1)$ estimate of $(\omega,\theta)$.
\begin{lemma}
Assume $\omega_0\in \dot{W}^{1,p}\cap H^s$ and $\theta_0\in \dot{W}^{1,p}\cap H^s$~$(2<p<\infty, s>1)$, then the solution $(\omega,\theta)$ satisfies
\begin{equation}
\begin{split}
\||D|^s\omega\|_{L^\infty_t(L^2)}^2+\||D|^s\theta\|_{L^\infty_t(L^2)}^2+\|\partial_1|D|^s\omega\|_{L^2_t(L^2)}^2\leq C.
\end{split}
\end{equation}
\end{lemma}
\begin{proof}
Taking $L^2$ inner product with $(|D|^s\omega,|D|^s\theta)$ and adding them up, we have
\begin{equation}
\label{Hs--omega-theta}
\begin{split}
&\quad\frac12\frac{d}{dt}(\||D|^s\omega(t)\|_{L^2}^2+\||D|^s\theta(t)\|_{L^2}^2)+\|\partial_1|D|^s\omega\|_{L^2}^2\\
&=\int_{\R^2}\partial_1|D|^s\theta |D|^s\omega~dx
-\int_{\R^2}[|D|^s,u\cdot \nabla] \omega |D|^s\omega~dx
-\int_{\R^2}[|D|^s,u\cdot \nabla] \theta |D|^s\theta~dx\\
&\triangleq K_1+K_2+K_3.
\end{split}
\end{equation}
For $K_1$, after integration by part and Young's inequality,
\begin{equation}
\label{K1}
K_1\leq \frac12\||D|^s\theta\|_{L^2}+\frac12\|\partial_1|D|^s\omega\|_{L^2}.
\end{equation}
For $K_2$, by virtue of the H\"older inequality and \eqref{kato-ponce1} in Lemma \ref{kato-ponce},
\begin{equation*}
\begin{split}
K_2&\leq \|[|D|^s,u\cdot \nabla] \omega\|_{L^2}\||D|^s\omega\|_{L^2}\\
&\leq C (\|\nabla u\|_{L^\infty}\||D|^{s-1}\nabla\omega\|_{L^2}+\||D|^s u\|_{L^p}\|\nabla\omega\|_{L^{p'}})\||D|^s\omega\|_{L^2},
\end{split}
\end{equation*}
with $\frac1p+\frac{1}{p'}=\frac12$, $p\in(2,\infty)$. By interpolation
\begin{equation*}
\begin{split}
\|f\|_{L^p}\leq C\|f\|_{L^2}^{\frac2p}\|\nabla f\|_{L^2}^{1-\frac{2}{p}},
\end{split}
\end{equation*}
then we can obtain
\begin{equation*}
\begin{split}
\||D|^s u\|_{L^p}\leq C\||D|^s u\|_{L^2}^{\frac2p}\|\nabla |D|^s u\|_{L^2}^{1-\frac{2}{p}}\leq C(\|u\|_{L^2}+\||D|^s\omega\|_{L^2}).
\end{split}
\end{equation*}
Thus we have
\begin{equation}
\label{K2}
\begin{split}
K_2\leq C (\|\nabla u\|_{L^\infty}+\|\nabla\omega\|_{L^{p'}})\times(\||D|^s\omega\|_{L^2}^2+1).
\end{split}
\end{equation}
Similarly,
\begin{equation}
\label{K3}
\begin{split}
K_3\leq C (\|\nabla u\|_{L^\infty}+\|\nabla\theta\|_{L^{p'}})\times(\||D|^s\omega\|_{L^2}^2+\||D|^s\theta\|_{L^2}^2+1).
\end{split}
\end{equation}
Inserting \eqref{K1}, \eqref{K2} and \eqref{K3} into \eqref{Hs--omega-theta}, making use of Lemma \ref{w1p-omega-theta} and Gr\"onwall's Lemma, we can
deduce
\begin{equation*}
\begin{split}
\||D|^s\omega\|_{L^\infty_t(L^2)}^2+\||D|^s\theta\|_{L^\infty_t(L^2)}^2+\|\partial_1|D|^s\omega\|_{L^2_t(L^2)}^2\leq C,
\end{split}
\end{equation*}
which complete the proof of this lemma.
\end{proof}

\begin{subsection}{A priori estimates for the higher order striated regularity}
~\\
In this subsection, we will give the higher order estimates of the
vector field $X$. The first  lemma asserts the $H^1$ estimate of $X$.
\begin{lemma}
Let $\omega_0\in H^1, \theta_0\in H^1$ and $X_0\in H^1$, then we have
\begin{equation}
\label{H1X}
\begin{split}
\|\nabla X\|_{L^\infty_t({L^2})}^2
\leq C.
\end{split}
\end{equation}
\end{lemma}
\begin{proof}
Applying $\partial_k~(k=1,2)$ to the first equation of \eqref{X-eq}, we can obtain $\partial_k X$ satisfies
\begin{equation}
\label{nablaX}
\begin{split}
\partial_t \partial_kX+u\cdot\nabla \partial_kX+ \partial_ku\cdot\nabla X=\partial_k\partial_X u
\end{split}
\end{equation}
Multiplying $\partial_k X$ to \eqref{nablaX} and integrating over $\R^2$ with respect to $x$, we have
\begin{equation}
\label{H^1X}
\begin{split}
\frac12\frac{d}{dt}\|\partial_kX(t)\|_{L^2}^2
&=\int_{\R^2}\partial_k\partial_Xu\cdot \partial_kX~dx
-\int_{\R^2}\partial_ku\cdot \nabla X\cdot \partial_kX~dx\\
&=\int_{\R^2}\partial_kX\cdot\nabla u\cdot \partial_kX~dx+\int_{\R^2}X\cdot\nabla \partial_ku\cdot \partial_kX~dx\\
&\quad-\int_{\R^2}\partial_ku\cdot \nabla X\cdot \partial_kX~dx\\
&\triangleq B_1+B_2+B_3.
\end{split}
\end{equation}
By H\"older inequality, $B_1$ can be bounded by
\begin{equation}
\label{B1}
\begin{split}
B_1= \int_{\R^2}\partial_kX\cdot\nabla u\cdot \partial_kX~dx\leq\|\nabla u\|_{L^\infty}\|\nabla X\|_{L^2}^2.
\end{split}
\end{equation}
Similarly,
\begin{equation}
\label{B3}
\begin{split}
B_3= -\int_{\R^2}\partial_ku\cdot \nabla X\cdot \partial_kX~dx\leq\|\nabla u\|_{L^\infty}\|\nabla X\|_{L^2}^2.
\end{split}
\end{equation}
Then by virtue of anisotropic H\"older inequality,
\begin{equation}
\label{B2}
\begin{split}
B_2&= \int_{\R^2}X\cdot\nabla \partial_ku\cdot \partial_kX~dx\\
&\leq C\|X\|_{L^\infty_{x_2}(L^2_{x_1})}\|\partial_k\nabla u\|_{L^2_{x_2}(L^\infty_{x_1})}
\|\partial_kX\|_{L^2(\R^2)}\\
&\leq C\|X\|_{L^2}^\frac12\|\partial_2X\|_{L^2}^\frac12\|\nabla \omega\|_{L^2}^\frac12\|\partial_1\nabla \omega\|_{L^2}^\frac12
\|\partial_kX\|_{L^2}\\
&\leq C(\|\nabla \omega\|_{L^2}+\|\partial_1\nabla \omega\|_{L^2})\times(\|X\|_{L^2}^2+\|\nabla X\|_{L^2}^2).
\end{split}
\end{equation}
After substituting \eqref{B1}, \eqref{B2} and \eqref{B3} into \eqref{H^1X}, we find that
\begin{equation}
\label{H^1estimateX}
\begin{split}
\frac{d}{dt}\|\nabla X(t)\|_{L^2}^2
\leq C(\|\nabla u\|_{L^\infty}+\|\nabla \omega\|_{L^2}+\|\partial_1\nabla \omega\|_{L^2})\times(\|X\|_{L^2}^2+\|\nabla X\|_{L^2}^2).
\end{split}
\end{equation}
Combining the estimates \eqref{X-Lr-estimate} and \eqref{H^1estimateX}, using Gronwall's Lemma and by Lemma \ref{lipschitz-information} and Lemma \ref{H1--omega}, we can deduce
\begin{equation*}
\begin{split}
\|\partial_X u\|_{L^\infty_t({L^2})}^2+
\|\partial_1\partial_X u\|_{L^2_t({L^2})}^2+
\|\nabla X\|_{L^\infty_t({L^2})}^2
\leq C,
\end{split}
\end{equation*}
which complete the proof of this lemma.
\end{proof}

The next lemma shows the $H^s~(s>1)$ estimate for $X$.
\begin{lemma}
Assume $\omega_0\in\dot{ W}^{1,p}\cap H^s, \theta_0\in \dot{ W}^{1,p}\cap H^s$, $X_0\in H^s$~$(2<p<\infty, s>1)$, then we have
\begin{equation}
\label{HsX}
\begin{split}
\||D|^sX\|_{L^\infty_t(L^2)}^2\leq C.
\end{split}
\end{equation}
\end{lemma}
\begin{proof}
Applying $|D|^s$ to the first equation of \eqref{X-eq}, making use of the definition of commutator, we can obtain $|D|^sX$ satisfies the follow equation
\begin{equation}
\label{Lamda^sX-eq}
\begin{split}
\partial_t|D|^sX+u\cdot\nabla|D|^sX=-[|D|^s,u\cdot\nabla]X+|D|^s(X\cdot\nabla u).
\end{split}
\end{equation}
Taking $L^2$ inner product with $|D|^sX$,
\begin{equation}
\label{Hs--X-1}
\begin{split}
\frac12\frac{d}{dt}(\||D|^sX(t)\|_{L^2}^2
&=-\int_{\R^2}[|D|^s,u\cdot\nabla]X\cdot|D|^sX~dx
+\int_{\R^2}|D|^s(X\cdot\nabla u)\cdot|D|^sX~dx\\
&\triangleq M_1+M_2.
\end{split}
\end{equation}
For $M_1$, by H\"older inequality and \eqref{kato-ponce1} in Lemma \ref{kato-ponce},
\begin{equation*}
\begin{split}
M_1&\leq \|[|D|^s,u\cdot \nabla] X\|_{L^2}\||D|^sX\|_{L^2}\\
&\leq C (\|\nabla u\|_{L^\infty}\||D|^{s-1}\nabla X\|_{L^2}+\||D|^s u\|_{L^p}\|\nabla X\|_{L^{p'}})\||D|^s X\|_{L^2},
\end{split}
\end{equation*}
with $\frac1p+\frac{1}{p'}=\frac12$, $p\in(2,\infty)$. Choosing $p$ such that
\begin{equation*}
\begin{split}
\|\nabla X\|_{L^{p'}}\leq C\||D|^{s}X\|_{L^2},
\end{split}
\end{equation*}
and noticing that by interpolation
\begin{equation*}
\begin{split}
\||D|^s u\|_{L^p}\leq C\||D|^{s}u\|_{L^2}^\frac2p\||D|^{s+1}u\|_{L^2}^{1-\frac{2}{p}}\leq C(\|u\|_{L^2}+\||D|^{s}\omega\|_{L^2}),
\end{split}
\end{equation*}
then we have
\begin{equation}
\label{M1}
\begin{split}
M_1\leq C (\|\nabla u\|_{L^\infty}+\||D|^s\omega\|_{L^{2}}+1)\||D|^sX\|_{L^2}^2.
\end{split}
\end{equation}
Next we estimate $M_2$, making use of the H\"older inequality and inequality \eqref{kato-ponce2} in Lemma \ref{kato-ponce},
\begin{equation*}
\begin{split}
M_2&\leq \||D|^s(X\cdot \nabla u)\|_{L^2}\||D|^sX\|_{L^2}\\
&\leq C (\|X\|_{L^\infty}\||D|^{s}\nabla u\|_{L^2}+\||D|^s X\|_{L^2}\|\nabla u\|_{L^{\infty}})\||D|^s X\|_{L^2}.
\end{split}
\end{equation*}
By Sobolev embedding,
\begin{equation*}
\begin{split}
\|X\|_{L^\infty}\leq C\||D|^s X\|_{L^2},~~~~~~~~\text{for}~s>1.
\end{split}
\end{equation*}
Thus we have
\begin{equation}
\label{M2}
\begin{split}
M_2\leq C (\||D|^{s}\omega\|_{L^2}+\|\nabla u\|_{L^{\infty}})\||D|^s X\|_{L^2}^2.
\end{split}
\end{equation}
Inserting \eqref{M1} and \eqref{M2} into \eqref{Hs--X-1}, using Gr\"onwall's Lemma, we can
deduce
\begin{equation*}
\begin{split}
\||D|^sX\|_{L^\infty_t(L^2)}^2\leq C,
\end{split}
\end{equation*}
which complete the proof of this lemma.
\end{proof}
\subsection{The temperature patch problem}
This subsection is devoted to the proof of Corollary \ref{cor-2}. Because most of the proof is the same to Corollary \ref{cor-1}, here we just need to verify the inequality \eqref{distance-temperature-patch}. 
Choosing arbitrary two points that $x_1\in D_{\varepsilon}^-$, $x_2\in D_{\varepsilon}^+$, consider the difference
\begin{equation}
\label{dis-1}
|\psi(x_1,t)-\psi(x_2,t)|\leq \|\nabla\psi\|_{L^\infty} |x_1-x_2|, \quad\text{for~~any~~}t>0.
\end{equation}
Noticing that from \eqref{flow-map}, we have
\begin{equation}
\label{estimate-flow}
\|\nabla\psi\|_{L^\infty}\leq e^{\int_0^t\|\nabla u(\tau)\|_{L^\infty}~d\tau}.
\end{equation}
Then inserting the estimate \eqref{estimate-flow} into \eqref{dis-1} and taking infimum of $x_1, x_2$, we can obtain 
\begin{equation*}
|d(t)|\leq 2\varepsilon e^{\int_0^t\|\nabla u(\tau)\|_{L^\infty}d\tau}.
\end{equation*}
which is the desired bounded \eqref{distance-temperature-patch}.

\end{subsection}

\end{section}

\begin{appendices}

\section{Appendix}

The goal of this appendix is to give the proof of
Lemma \ref{lemA.1} and Lemma \ref{tran}.
\begin{proof}[Proof of Lemma  \ref{lemA.1}]
The proof of \eqref{lem1eq1} can be found in \cite{DP1} which used the standard Bony's decomposition (see \cite{Chemin-perfect,BCD}). Here we focus on proving \eqref{lem1eq} using the anisotropic idea. Firstly, we divide the first term of \eqref{lem1eq} into two terms,
\begin{equation*}
\begin{split}
-\int_{\R^2}\Delta_q(u\cdot \nabla f)\Delta_q f~dx&=-\int_{\R^2}\Delta_q(u^1\partial_1 f)\Delta_q f~dx
-\int_{\R^2}\Delta_q(u^2\partial_2 f)\Delta_q f~dx\\
&\triangleq P+Q.
\end{split}
\end{equation*}
For $P$, by Bony's decomposition, we can divide it into the following three terms,
\begin{equation}
\label{eqlem21}
\begin{split}
&-\int_{\R^2}\Delta_q(u^1 \partial_1 f)\Delta_q f~dx\\
=&-\sum_{|k-q|\leq2}\int_{\R^2}\Delta_q(S_{k-1}u^1 \Delta_k\partial_1 f)\Delta_q f~dx\\
&-\sum_{|k-q|\leq2}\int_{\R^2}\Delta_q(\Delta_ku^1 S_{k-1}\partial_1  f)\Delta_q f~dx\\
&-\sum_{k\geq q-1}\sum_{|k-l|\leq1}\int_{\R^2}\Delta_q(\Delta_ku^1  \Delta_l\partial_1 f)\Delta_q f~dx\\
\triangleq& P_1+P_2+P_3.
\end{split}
\end{equation}
For $P_1$, we can rewrite it as
\begin{equation*}
\begin{split}
P_1&=-\sum_{|q-k|\leq2}\int_{\R^2}\Delta_q(S_{k-1}u^1\partial_1\Delta_k f)\Delta_q f~dx\\
&=-\sum_{|q-k|\leq2}\int_{\R^2} [\Delta_q, S_{k-1}u^1\partial_1]\Delta_k f \Delta_q f~dx\\
&\quad-\sum_{|q-k|\leq2}\int_{\R^2}S_{k-1}u^1\partial_1\Delta_q\Delta_k f \Delta_q f~dx\\
&=-\sum_{|q-k|\leq2}\int_{\R^2}[\Delta_q, S_{k-1}u^1\partial_1]\Delta_k f \Delta_q f~dx\\
&\quad-\sum_{|q-k|\leq2}\int_{\R^2}(S_{k-1}u^1-S_qu^1)\partial_1\Delta_q \Delta_kf\Delta_q f~dx\\
&\quad-\int_{\R^2}{S}_{q}u^1\partial_1\Delta_q f\Delta_q f~dx\\
&\triangleq P_{11}+P_{12}+P_{13},
\end{split}
\end{equation*}
where we have used the fact $\sum_{|q-k|\leq2}\partial_1\Delta_q\Delta_k f=\Delta_q f$. For $P_{11}$, by H\"older inequality,
\begin{equation*}
\begin{split}
|P_{11}|&\leq\sum_{|q-k|\leq2}\bigg|\int_{\R^2}[\Delta_q, S_{k-1}u^1\partial_1]\Delta_k f\Delta_k f~dx\bigg|\\
&\leq C\sum_{|q-k|\leq2}\|[\Delta_q, S_{k-1}u^1\partial_1]\Delta_k f\|_{L^2}\|\Delta_qf\|_{L^2}.
\end{split}
\end{equation*}
According to the definition of $\Delta_q$,
\begin{equation*}
\begin{split}
[\Delta_q, S_{k-1}u^1\partial_1]\Delta_k f&=\int_{\R^2}\phi_q(x-y)(S_{k-1}u^1(y)\partial_1\Delta_k f(y))~dy\\
&\quad\quad-S_{k-1}u^1(x)\int_{\R^d}\phi_q(x-y)\partial_1\Delta_k f(y)~dy\\
&=\int_{\R^2}\phi_q(x-y)(S_{k-1}u^1(y)-S_{k-1}u^1(x))\partial_1\Delta_k f(y)~dy\\
&=\int_{\R^2}\phi_q(x-y)\int_0^1(y-x)\cdot\nabla S_{k-1}u^1(sy+(1-s)x)~ds\partial_1\Delta_k f(y)~dy\\
&=\int_{\R^2}\int_0^1\phi_q(z)z\cdot\nabla S_{k-1}u^1(x-sz)\partial_1\Delta_k f(x-z)~ds dz,\\
\end{split}
\end{equation*}
where $\phi_j(x)\triangleq2^{jd}\mathcal{F}^{-1}(\phi)(2^jx)$.
Thus we have by H\"older inequality and Bernstein inequality,
\begin{equation*}
\begin{split}
\|[\Delta_q, S_{k-1}u^1\partial_1]\Delta_k f\|_{L^2}
&=\bigg\|\int_{\R^2}\int_0^1\phi_q(z)z\cdot\nabla S_{k-1}u^1(x-sz)\partial_1\Delta_k f(x-z)~ds dz\bigg\|_{L^2}\\
&\leq C\int_{\R^2}\big|\phi_q(z)\big||z|~dz\|\nabla{S}_{k-1}u^1(x-sz)\|_{L^\infty}\|\partial_1\Delta_k f(x-z)\|_{L^2}\\
&\leq C\int_{\R^2}\big|\phi_q(z)\big||z|~dz\|\nabla S_{k-1}u^1\|_{L^\infty}\|\partial_1\Delta_k f\|_{L^2}\\
&\leq C2^{-q}2^k\|\nabla S_{k-1}u^1\|_{L^2}\|\partial_1\Delta_k f\|_{L^2}\\
&\leq C2^{k-q}\|\omega\|_{L^2}\|\partial_1\Delta_k f\|_{L^2}.
\end{split}
\end{equation*}
Then we obtain
\begin{equation*}
\begin{split}
|P_{11}|&\leq C\sum_{|q-k|\leq2}\|[\Delta_q, S_{k-1}u^1\partial_1]\Delta_k f\|_{L^2}\|\Delta_qf\|_{L^2}\\
&\leq C\sum_{|q-k|\leq2}2^{k-q}\|\omega\|_{L^2}\|\partial_1\Delta_k f\|_{L^2}\|\Delta_q f\|_{L^2}\\
&\leq Cb_q2^{2qs}\|\omega\|_{L^2}\|f\|_{H^s}\|\partial_1f\|_{H^s}.
\end{split}
\end{equation*}
For $P_{12}$, by H\"older inequality and Bernstein inequality,
\begin{equation*}
\begin{split}
|P_{12}|&=\sum_{|q-k|\leq2}\bigg|\int_{\R^2} ((S_{k-1}u^1-S_{q}u^1)\partial_1\Delta_q\Delta_k f)\Delta_q f~dx\bigg|\\
&\leq C\sum_{|q-k|\leq2}\|(S_{k-1}u^1-S_{q}u^1)\partial_1\Delta_q \Delta_kf\|_{L^1}\|\Delta_qf\|_{L^\infty}\\
&\leq C\sum_{|q-k|\leq2}\|\Delta_k u^1\|_{L^2}\|\Delta_q \Delta_k\partial_1 f\|_{L^2}2^q\|\Delta_qf\|_{L^2}.
\end{split}
\end{equation*}
For the case $k=-1$, by Bernstein inequality,
\begin{equation*}
\begin{split}
|P_{12}|
&\leq C\|\Delta_{-1} u^1\|_{L^2}2^{-1}\|\Delta_q \Delta_{-1} f\|_{L^2}2^{-1}\|\Delta_qf\|_{L^2}\\
&\leq Cb_q2^{2qs}\|u^1\|_{L^2}\|f\|_{H^s}^2.
\end{split}
\end{equation*}
For the case $k\geq0$, by Bernstein inequality,
\begin{equation*}
\begin{split}
|P_{12}|&\leq C\sum_{|q-k|\leq2}2^{-k}\|\nabla\Delta_k u^1\|_{L^2}\|\Delta_q \Delta_k\partial_1 f\|_{L^2}2^{q}\|\Delta_qf\|_{L^2}\\
&\leq C\sum_{|q-k|\leq2}2^{-k}\|\omega\|_{L^2}2^q\|\Delta_q\partial_1  f\|_{L^2}\|\Delta_qf\|_{L^2}\\
&\leq Cb_q2^{2qs}\|\omega\|_{L^2}\|f\|_{H^s}\|\partial_1f\|_{H^s}.
\end{split}
\end{equation*}
Thus,
\begin{equation}
\label{eqlem22}
\begin{split}
|P_{1}|\leq Cb_q2^{2qs}(\|u\|_{L^2}+\|\omega\|_{L^2})(\|f\|_{H^s}^2+\|f\|_{H^s}\|\partial_1f\|_{H^s}).
\end{split}
\end{equation}
For $P_2$, we can bound it by H\"older inequality that
\begin{equation*}
\begin{split}
|P_2|\leq C\sum_{|q-k|\leq2}\|\Delta_k u^1\|_{L^2}\|\partial_1 S_{k-1} f\|_{L^\infty}\|\Delta_qf\|_{L^2}.
\end{split}
\end{equation*}
Applying Bernstein inequality, similar as $P_{12}$,
\begin{equation}
\label{eqlem23}
\begin{split}
|P_2|&\leq C\sum_{|q-k|\leq2}\|\Delta_k u^1\|_{L^2}\sum_{m\leq k-2}2^m\|\Delta_m \partial_1f\|_{L^2}\|\Delta_qf\|_{L^2}\\
&\leq C \sum_{|q-k|\leq2}\|\Delta_k u^1\|_{L^2}\sum_{m\leq q-2}2^{m}\|\Delta_m f\|_{L^2}\|\Delta_qf\|_{L^2}\\
&\leq C \sum_{|q-k|\leq2}2^q\|\Delta_k u^1\|_{L^2}\sum_{m\leq q-2}2^{m-q}\|\Delta_m f\|_{L^2}\|\Delta_qf\|_{L^2}\\
&\leq C 2^{-qs}\sum_{|q-k|\leq2}2^{q-k}2^k\|\Delta_k u^1\|_{L^2}\sum_{m\leq q-2}2^{(m-q)(1-s)}2^{ms}\|\Delta_m f\|_{L^2}\|\Delta_qf\|_{L^2}\\
&\leq C b_q 2^{-2qs} (\|u^1\|_{L^2}+\|\omega\|_{L^2})\|f\|_{H^s}\|\partial_1f\|_{H^s},\\
\end{split}
\end{equation}
where we have used discrete Young's inequality in the last step.\\
Next we estimate $P_3$. By H\"older inequality and Bernstein inequality,
\begin{equation}
\label{eqlem24}
\begin{split}
|P_{3}|&\leq \bigg|\sum_{k\geq q-1}\sum_{|k-l|\leq1}\int_{\R^2} \Delta_q(\Delta_ku^1\partial_1\Delta_q f)\Delta_qf~dx\bigg|\\
&\leq C \sum_{k\geq q-1}\sum_{|k-l|\leq1}\|\Delta_q(\Delta_ku^1\Delta_l\partial_1 f)\|_{L^1}\|\Delta_qf\|_{L^\infty}\\
&\leq C 2^q\sum_{k\geq q-1}2^{-k}2^k\|\Delta_ku^1\|_{L^2}\|\Delta_k\partial_1f\|_{L^2}\|\Delta_qf\|_{L^2}\\
&\leq C2^{-qs} \sum_{k\geq q-1} 2^{(q-k)(1+s)}2^{ks}\|\Delta_k\partial_1f\|_{L^2}(\|u^1\|_{L^2}+\|\omega\|_{L^2})\|\Delta_qf\|_{L^2}\\
&\leq C b_q 2^{-2qs}(\|u^1\|_{L^2}+\|\omega\|_{L^2})\|f\|_{H^s}\|\partial_1f\|_{H^s},\\
\end{split}
\end{equation}
where discrete Young's inequality have been used in the last two line.\\
For $Q$, we can also divide it into three parts,
\begin{equation}
\label{eqlem31}
\begin{split}
-\int_{\R^2}\Delta_q(u^2 \partial_2 f)\Delta_q f~dx
=Q_1+Q_2+Q_3,
\end{split}
\end{equation}
with
\begin{equation*}
Q_1=-\sum_{|k-q|\leq2}\int_{\R^2}\Delta_q(S_{k-1}u^2 \Delta_k\partial_2 f)\Delta_q f~dx,
\end{equation*}
\begin{equation*}
Q_2=-\sum_{|k-q|\leq2}\int_{\R^2}\Delta_q(\Delta_ku^2 S_{k-1}\partial_2  f)\Delta_q f~dx
\end{equation*}
and
\begin{equation*}
Q_3-\sum_{k\geq q-1}\sum_{|k-l|\leq1}\int_{\R^2}\Delta_q(\Delta_ku^2  \Delta_l\partial_2 f)\Delta_q f~dx.
\end{equation*}
Similar as $P_1$, we can rewrite $Q_1$ as
\begin{equation*}
\begin{split}
Q_1&=-\sum_{|q-k|\leq2}\int_{\R^2}[\Delta_q, S_{k-1}u^2\partial_2]\Delta_k f \Delta_q f~dx\\
&\quad-\sum_{|q-k|\leq2}\int_{\R^2}(S_{k-2}u^1-S_qu^1)\partial_2\Delta_q \Delta_kf\Delta_q f~dx\\
&\quad-\int_{\R^2}{S}_{q}u^2\partial_2\Delta_q f\Delta_q f~dx\\
&\triangleq Q_{11}+Q_{12}+Q_{13}.
\end{split}
\end{equation*}
Here we should notice that $P_{13}+Q_{13}=0$ because of the divergence free condition of $u$, so we do not need to estimate these two terms.\\
For $Q_{11}$, by H\"older inequality,
\begin{equation*}
\begin{split}
|Q_{11}|&\leq\sum_{|q-k|\leq2}\bigg|\int_{\R^2}[\Delta_q, S_{k-1}u^1\partial_1]\Delta_k f\Delta_k f~dx\bigg|\\
&\leq C\sum_{|q-k|\leq2}\|[\Delta_q, S_{k-1}u^1\partial_1]\Delta_k f\|_{L^2}\|\Delta_qf\|_{L^2}.
\end{split}
\end{equation*}
According to the definition of $\Delta_q$ and similar as $P_{11}$,
\begin{equation*}
\begin{split}
[\Delta_q, S_{k-1}u^1\partial_1]\Delta_k f&=\int_{\R^2}\int_0^1\phi_q(z)z\cdot\nabla S_{k-1}u^2(x-sz)\partial_2\Delta_k f(x-z)~ds dz.\\
\end{split}
\end{equation*}
Making use of the anisotropic H\"older inequality and Bernstein inequality,
\begin{equation*}
\begin{split}
&~\quad\|[\Delta_q, S_{k-1}u^2\partial_2]\Delta_k f\|_{L^2}\\
&=\bigg\|\int_{\R^2}\int_0^1\varphi_q(z)z\cdot\nabla S_{k-1}u^2(x-sz)\partial_2\Delta_k f(x-z)~ds dz\bigg\|_{L^2}\\
&\leq C\int_{\R^2}\big|\varphi_q(z)\big||z|~dz\|\nabla{S}_{k-1}u^2(x-sz)\|_{L^\infty_{x_2}(L^2_{x_1})}\|\partial_2\Delta_k f(x-z)\|_{L^2_{x_2}(L^\infty_{x_1})}\\
&\leq C2^{-q}\|\nabla S_{k-1}u^2\|_{L^2}^{\frac12}\|\partial_2\nabla S_{k-1}u^2\|_{L^2}^{\frac12}\|\partial_2\Delta_k f\|_{L^2}^{\frac12}
\|\partial_1\partial_2\Delta_k f\|_{L^2}^{\frac12}.
\end{split}
\end{equation*}
Noticing that by Biot-Savart law  $u^2=\partial_1\Delta^{-1}\omega$, and combining with the boundedness of Riesz transform in $L^2$,
\begin{equation*}
\begin{split}
\quad\|[\Delta_q, S_{k-1}u^2\partial_2]\Delta_k f\|_{L^2}
&\leq C2^{k-q}\|\omega\|_{L^2}^{\frac12}\|\partial_2\nabla \partial_1\Delta^{-1}\omega\|_{L^2}^{\frac12}\|\Delta_k f\|_{L^2}^{\frac12}
\|\partial_1\Delta_k f\|_{L^2}^{\frac12}\\
&\leq C2^{k-q}\|\omega\|_{L^2}^{\frac12}\|\partial_1\omega\|_{L^2}^{\frac12}
\|\Delta_k f\|_{L^2}^{\frac12}
\|\partial_1\Delta_k f\|_{L^2}^{\frac12}.
\end{split}
\end{equation*}
Then $Q_{11}$ is bounded by
\begin{equation*}
\begin{split}
|Q_{11}|&\leq C\sum_{|q-k|\leq2}\|[\Delta_q, S_{k-1}u^2\partial_2]\Delta_k f\|_{L^2}\|\Delta_qf\|_{L^2}\\
&\leq C\sum_{|q-k|\leq2}2^{k-q}\|\omega\|_{L^2}^{\frac12}\|\partial_1\omega\|_{L^2}^{\frac12}
\|\Delta_k f\|_{L^2}^{\frac12}
\|\partial_1\Delta_k f\|_{L^2}^{\frac12}\|\Delta_qf\|_{L^2}\\
&\leq Cb_q2^{-2qs}\|\omega\|_{L^2}^{\frac12}\|\partial_1\omega\|_{L^2}^{\frac12}\|f\|_{H^s}^{\frac32}\|\partial_1f\|_{H^s}^{\frac12}.\\
\end{split}
\end{equation*}
For $Q_{12}$, by the anisotropic H\"older inequality and interpolation inequality,
\begin{equation*}
\begin{split}
|Q_{12}|&=\sum_{|q-k|\leq2}\bigg|\int_{\R^2} ((S_{k-1}u^2-S_{q}u^2)\partial_2\Delta_q\Delta_k f)\Delta_q f~dx\bigg|\\
&\leq C\sum_{|q-k|\leq2}\|(S_{k-1}u^2-S_{q}u^2)\partial_2\Delta_q \Delta_kf\|_{L^2}\|\Delta_qf\|_{L^2}\\
&\leq C\sum_{|q-k|\leq2}\|\Delta_k u^2\|_{L^\infty_{x_2}(L^2_{x_1})}\|\Delta_q \Delta_k\partial_2 f\|_{L^2_{x_2}(L^\infty_{x_1})}\|\Delta_qf\|_{L^2}\\
&\leq C\sum_{|q-k|\leq2}\|\Delta_k u^2\|_{L^2}^{\frac12}\|\Delta_k \partial_2u^2\|_{L^2}^{\frac12}\|\Delta_q \Delta_k\partial_2 f\|_{L^2}^{\frac12}\|\Delta_q \Delta_k\partial_1\partial_2 f\|_{L^2}^{\frac12}\|\Delta_qf\|_{L^2}\\
\end{split}
\end{equation*}
For the case $k=-1$, by Bernstein inequality,
\begin{equation*}
\begin{split}
|Q_{12}|
&\leq C\|\Delta_{-1} u^2\|_{L^2}\|\Delta_q \Delta_{-1} f\|_{L^2}\|\Delta_qf\|_{L^2}\\
&\leq Cb_q2^{-2qs}\|u^2\|_{L^2}\|f\|_{H^s}^2.
\end{split}
\end{equation*}
For the case $k\geq0$, by Bernstein inequality and the relation $u^2=\partial_1\Delta^{-1}\omega$,
\begin{equation*}
\begin{split}
|Q_{12}|&\leq C\sum_{|q-k|\leq2}2^{q-k}\|\nabla\Delta_k u^2\|_{L^2}^\frac12\|\nabla\Delta_k \partial_1\Delta^{-1}\omega\|_{L^2}^\frac12\|\Delta_q \partial_1 f\|_{L^2}^{\frac12}\|\Delta_qf\|_{L^2}^{\frac32}\\
&\leq C\sum_{|q-k|\leq2}2^{q-k}\|\omega\|_{L^2}^\frac12\|\partial_1\omega\|_{L^2}^\frac12
\|\Delta_q\partial_1  f\|_{L^2}^\frac12\|\Delta_qf\|_{L^2}^\frac32\\
&\leq Cb_q2^{-2qs}\|\omega\|_{L^2}^\frac12\|\partial_1\omega\|_{L^2}^\frac12\|f\|_{H^s}\|\partial_1f\|_{H^s}.
\end{split}
\end{equation*}
Thus,
\begin{equation}
\label{eqlem32}
\begin{split}
|Q_{1}|\leq Cb_q2^{-2qs}(\|u\|_{L^2}+\|\omega\|_{L^2}^\frac12\|\partial_1\omega\|_{L^2}^\frac12)(\|f\|_{H^s}^2+\|f\|_{H^s}\|\partial_1f\|_{H^s}).
\end{split}
\end{equation}
Similar as $Q_{12}$, applying anisotropic H\"older inequality and Bernstein inequality, $Q_{2}$ can be bounded by
\begin{equation}
\label{eqlem33}
\begin{split}
|Q_2|&\leq C\sum_{|q-k|\leq2}\|\Delta_k u^2\|_{L^\infty_{x_2}(L^2_{x_1})}\|\partial_2 S_{k-1} f\|_{L^2_{x_2}(L^\infty_{x_1})}\|\Delta_qf\|_{L^2}\\
&\leq C\sum_{|q-k|\leq2}\|\Delta_k u^2\|_{L^2}^\frac12
\|\Delta_k \partial_2u^2\|_{L^2}^\frac12
\|\partial_2 S_{k-1} f\|_{L^2}^\frac12\|\partial_1\partial_2 S_{k-1} f\|_{L^2}^\frac12\|\Delta_qf\|_{L^2}\\
&\leq Cb_q2^{-2qs}\|u\|_{L^2}\|f\|_{H^s}^2+C\|\omega\|_{L^2}^\frac12\|\partial_1\omega\|_{L^2}^\frac12\bigg(\sum_{m\leq q-2}2^{m-q}\|\Delta_m f\|_{L^2}\bigg)^{\frac12}\\&\quad\quad\times\bigg(\sum_{n\leq q-2}2^{n-q}\|\Delta_n \partial_1f\|_{L^2}\bigg)^{\frac12}\|\Delta_qf\|_{L^2}\\
&\leq C b_q2^{-2qs} (\|u\|_{L^2}+\|\omega\|_{L^2}^\frac12\|\partial_1\omega\|_{L^2}^\frac12)(\|f\|_{H^s}^2+\|f\|_{H^s}^\frac32\|\partial_1f\|_{H^s}^\frac12).\\
\end{split}
\end{equation}
Finally we estimate $Q_3$. By H\"older inequality and Bernstein inequality,
\begin{equation}
\label{eqlem34}
\begin{split}
|Q_{3}|&\leq \bigg|\sum_{k\geq q-1}\sum_{|k-l|\leq1}\int_{\R^2} \Delta_q(\Delta_ku^2\partial_2\Delta_q f)\Delta_qf~dx\bigg|\\
&\leq C \sum_{k\geq q-1}\sum_{|k-l|\leq1}\|\Delta_q(\Delta_ku^2\Delta_l\partial_2 f)\|_{L^1}\|\Delta_qf\|_{L^\infty}\\
&\leq C 2^q\sum_{k\geq q-1}\|\Delta_ku^2\|_{L^2}\|\Delta_k\partial_2f\|_{L^2}\|\Delta_qf\|_{L^2}\\
&\leq C 2^q\sum_{k\geq q-1}(\|u\|_{L^2}+\|\partial_1\omega\|_{L^2})2^{-2k}2^k\|\Delta_kf\|_{L^2}\|\Delta_qf\|_{L^2}\\
&\leq C b_q2^{-2qs}(\|u\|_{L^2}+\|\partial_1\omega\|_{L^2})\|f\|_{H^s}^2.\\
\end{split}
\end{equation}
Taking all these estimates into account, we can obtain
\begin{equation*}
\begin{split}
\label{eqlem1}
-\int_{\R^2}\Delta_q(u\cdot \nabla f)\Delta_q f~dx\leq&  C b_q2^{-2qs}(\|u\|_{L^2}+\|\omega\|_{L^2}+\|\partial_1\omega\|_{L^2})\\&\quad\quad\times(\|f\|_{H^s}^2+\|f\|_{H^s}^\frac12\|\partial_1f\|_{H^s}^\frac12+\|f\|_{H^s}^\frac32\|\partial_1f\|_{H^s}^\frac12),
\end{split}
\end{equation*}
which complete the proof of this lemma.

\end{proof}

\begin{lemma}[Losing regularity estimate for transport equation]
Let $\rho$ satisfies the transport equation
\begin{equation}
\left\{
\begin{array}{cc}
\begin{split}
&\partial_t \rho+u\cdot\nabla \rho=f,\\
&\rho(0,x)=\rho_0(x),
\end{split}
\end{array}
\right.
\end{equation}
where $\rho_0\in B^s_{2,r}$, $f\in L^1([0,T];B^s_{2,r})$ with $r\in[1,\infty]$. Here $v\in L^2$ is a divergence free
vector field and for some $V(t)\in L^1([0,T])$, $v$ satisfies
$$\sup_{N\geq0}\frac{\|\nabla S_N v(t)\|_{L^{\infty}}}{\sqrt{1+N}}\leq V(t). $$
Then for all $s>0$, $\varepsilon\in (0,s)$ and $t\in [0,T]$, we have the following estimate,
\begin{equation*}
\begin{split}
\|\rho(t)\|_{B^{s-\varepsilon}_{2,r}}&\leq C(T) \bigg(\|\rho_0\|_{B^{s}_{2,r}}+\int_0^T\|f(\tau)\|_{B^{s}_{2,r}}~d\tau\bigg)e^{\frac{C}{\varepsilon}\big(\int_0^TV(\tau)~d\tau\big)^2}.
\end{split}
\end{equation*}
\end{lemma}
\begin{proof}
The case $r=\infty$ has been shown in \cite{DP1}, here we just discuss $1\leq r<\infty$.
Applying $\Delta_q$ to \eqref{tran}, we obtain
\begin{equation}
\partial_t\Delta_q\rho+\Delta_q(v\cdot\nabla\rho)=\Delta f.
\end{equation}
Taking $L^2$ inner product with $\Delta_q\rho$,
\begin{equation}
\label{L2rho}
\begin{split}
\frac12\frac{d}{dt}\|\Delta_q\rho\|_{L^2}^2=-\int_{\R^2}\Delta_q(v\cdot\nabla\rho)\Delta_q\rho~dx+\int_{\R^2}\Delta_qf\Delta_q\rho~dx
\triangleq I+II.
\end{split}
\end{equation}
For $II$, by H\"older inequality,
\begin{equation}
\label{estimateII}
\begin{split}
II=\int_{\R^2}\Delta_qf\Delta_q\rho\leq \|\Delta_qf\|_{L^2}\|\Delta_q\rho\|_{L^2}.
\end{split}
\end{equation}
For $I$, along a similar argument as Lemma \ref{lemA.1}, we can divide it as
\begin{equation*}
\begin{split}
I=&-\int_{\R^2}\Delta_q(u\cdot \nabla \rho)\Delta_q \rho~dx\\
=&-\sum_{|k-q|\leq2}\int_{\R^2}\Delta_q(S_{k-1}u\cdot \Delta_k\nabla \rho)\Delta_q \rho~dx\\
&-\sum_{|k-q|\leq2}\int_{\R^2}\Delta_q(\Delta_ku\cdot \nabla S_{k-1} \rho)\Delta_q \rho~dx\\
&-\sum_{k\geq q-1}\sum_{|k-l|\leq1}\int_{\R^2}\Delta_q(\Delta_ku\cdot \nabla \Delta_l \rho)\Delta_q \rho~dx\\
\triangleq& L_1+L_2+L_3.
\end{split}
\end{equation*}
For $L_1$, we can rewrite it as
\begin{equation*}
\begin{split}
L_1&=-\sum_{|q-k|\leq2}\int_{\R^2}[\Delta_q, S_{k-1}u\cdot\nabla]\Delta_k \rho \Delta_q \rho~dx\\
&\quad-\sum_{|q-k|\leq2}\int_{\R^2}(S_{k-1}u-S_qu)\cdot\nabla\Delta_q \Delta_k\rho\Delta_q \rho~dx\\
&\quad-\int_{\R^2}{S}_{q}u\cdot\nabla\Delta_q \rho\Delta_q f~dx\\
&\triangleq L_{11}+L_{12}+L_{13},
\end{split}
\end{equation*}
According to divergence-free condition of $u$, it is not difficult to find that $L_{13}=0$. For $L_{11}$, by H\"older inequality,
\begin{equation*}
\begin{split}
|L_{11}|&\leq\sum_{|q-k|\leq2}\bigg|\int_{\R^2}[\Delta_q, S_{k-1}u\cdot\nabla]\Delta_k \rho\Delta_k \rho~dx\bigg|\\
&\leq C\sum_{|q-k|\leq2}\|[\Delta_q, S_{k-1}u\cdot\nabla]\Delta_k \rho\|_{L^2}\|\Delta_q\rho\|_{L^2}.
\end{split}
\end{equation*}
According to the definition of $\Delta_q$,
\begin{equation*}
\begin{split}
[\Delta_q, S_{k-1}u\cdot\nabla]\Delta_k \rho&=\int_{\R^2}\phi_q(x-y)(S_{k-1}u(y)\cdot\nabla\Delta_k \rho(y))~dy\\
&\quad\quad-S_{k-1}u(x)\cdot\int_{\R^d}\phi_q(x-y)\nabla\Delta_k \rho(y)~dy\\
&=\int_{\R^2}\phi_q(x-y)(S_{k-1}u(y)-S_{k-1}u(x))\cdot\nabla\Delta_k \rho(y)~dy\\
&=\int_{\R^2}\phi_q(x-y)\int_0^1(y-x)\cdot\nabla S_{k-1}u(sy+(1-s)x)~ds\cdot\nabla\Delta_k \rho(y)~dy\\
&=\int_{\R^2}\int_0^1\phi_q(z)z\cdot\nabla S_{k-1}u(x-sz)\cdot\nabla\Delta_k \rho(x-z)~ds dz.\\
\end{split}
\end{equation*}
Thus we have
\begin{equation*}
\begin{split}
\|[\Delta_q, S_{k-1}u\cdot\nabla]\Delta_k \rho\|_{L^2}
&=\bigg\|\int_{\R^2}\int_0^1\phi_q(z)z\cdot\nabla S_{k-1}u(x-sz)\cdot\nabla\Delta_k \rho(x-z)~ds dz\bigg\|_{L^2}\\
&\leq C\int_{\R^2}\big|\phi_q(z)\big||z|~dz\|\nabla{S}_{k-1}u(x-sz)\|_{L^\infty}\|\nabla\Delta_k \rho(x-z)\|_{L^2}\\
&\leq C2^{-q}\int_{\R^2}\big|\phi_q(z)\big||z|~dz\|\nabla S_{k-1}u\|_{L^\infty}\|\nabla\Delta_k \rho\|_{L^2}\\
&\leq C2^{-q}\|\nabla S_{k-1}u\|_{L^\infty}2^k\|\Delta_k \rho\|_{L^2}.
\end{split}
\end{equation*}
Then we obtain
\begin{equation*}
\begin{split}
|L_{11}|&\leq C\sum_{|q-k|\leq2}\|[\Delta_q, S_{k-1}u\cdot\nabla]\Delta_k \rho\|_{L^2}\|\Delta_q\rho\|_{L^2}\\
&\leq C\sum_{|q-k|\leq2}2^{k-q}\|\nabla S_{k-1}u\|_{L^\infty}\|\Delta_q \rho\|_{L^2}^2\\
&\leq C\sqrt{q}V(t)\|\Delta_q \rho\|_{L^2}^2\\
&\leq Cd_q2^{-\sigma q}\sqrt{q}V(t)\| \rho\|_{B^{\sigma}_{2,r}}\|\Delta_q \rho\|_{L^2},
\end{split}
\end{equation*}
where $d_q\in\ell^r$.\\
For $L_{12}$, by H\"older inequality,
\begin{equation*}
\begin{split}
|L_{12}|&=\sum_{|q-k|\leq2}\bigg|\int_{\R^2} ((S_{k-1}u-S_{q}u)\cdot\nabla\Delta_q\Delta_k \rho)\Delta_q \rho~dx\bigg|\\
&\leq C\sum_{|q-k|\leq2}2^{q-k}\|\nabla\Delta_k u\|_{L^\infty}\|\Delta_q\rho\|_{L^2}^2+\|u\|_{L^2}\|\Delta_q\rho\|_{L^2}^2\\
&\leq C(\sqrt{q+2}V(t)+\|u\|_{L^2})\|\Delta_q\rho\|_{L^2}^2\\
&\leq Cd_q2^{-\sigma q}(\sqrt{q+2}V(t)+\|u\|_{L^2})\| \rho\|_{B^{\sigma}_{2,r}}\|\Delta_q \rho\|_{L^2}.
\end{split}
\end{equation*}
For $L_2$, we can bound it by H\"older inequality that
\begin{equation*}
\begin{split}
|L_2|\leq C\sum_{|q-k|\leq2}\|\Delta_k u\|_{L^\infty}\|\nabla S_{k-1} \rho\|_{L^2}\|\Delta_q\rho\|_{L^2}.
\end{split}
\end{equation*}
According to Bernstein inequality,
\begin{equation*}
\begin{split}
|L_2|
&\leq C \sum_{|q-k|\leq2}\|\Delta_k u\|_{L^\infty}\sum_{m\leq q-2}2^{m}\|\Delta_m \rho\|_{L^2}\|\Delta_q\rho\|_{L^2}\\
&\leq C \sum_{|q-k|\leq2}2^q\|\Delta_k u\|_{L^\infty}\sum_{m\leq q-2}2^{m-q}\|\Delta_m \rho\|_{L^2}\|\Delta_q\rho\|_{L^2}\\
&\leq C \sum_{|q-k|\leq2}2^q\|\Delta_k u\|_{L^\infty}\sum_{m\leq q-2}2^{m-q}\|\Delta_m \rho\|_{L^2}\|\Delta_q\rho\|_{L^2}\\
&\leq C (\sqrt{q+2}V(t)+\|u\|_{L^2})\sum_{m\leq q-2}2^{m-q}\|\Delta_m \rho\|_{L^2}\|\Delta_q\rho\|_{L^2}\\
&\leq C d_q2^{-\sigma q}(\sqrt{q+2}V(t)+\|u\|_{L^2}) \|\rho\|_{B^{\sigma}_{2,r}}\|\Delta_q\rho\|_{L^2}.
\end{split}
\end{equation*}
Then we bound $L_3$. By H\"older inequality and Bernstein inequality,
\begin{equation*}
\begin{split}
|L_{3}|&\leq \bigg|\sum_{k\geq q-1}\sum_{|k-l|\leq1}\int_{\R^2} \Delta_q(\Delta_ku\cdot\nabla\Delta_q \rho)\|\Delta_q\rho\|_{L^2}~dx\bigg|\\
&\leq C \sum_{k\geq q-1}\sum_{|k-l|\leq1}\|\Delta_q\nabla\cdot(\Delta_ku\Delta_l \rho)\|_{L^2}\|\Delta_q\rho\|_{L^2}\\
&\leq C 2^q\sum_{k\geq q-1}\|\Delta_ku\|_{L^\infty}\|\Delta_k\rho\|_{L^2}\|\Delta_q\rho\|_{L^2}\\
&\leq C (\sqrt{q+2}V(t)+\|u\|_{L^2})\sum_{k\geq q-1}2^{q-k}\|\Delta_k\rho\|_{L^2}\|\Delta_q\rho\|_{L^2}\\
&\leq C d_q2^{-\sigma q}(\sqrt{q+2}V(t)+\|u\|_{L^2})\|\rho\|_{B^{\sigma}_{2,r}}\|\Delta_q\rho\|_{L^2}.\\
\end{split}
\end{equation*}
Thus we obtain $I$ can be bounded by
\begin{equation}
\label{estimateI}
\begin{split}
I\leq Cd_q2^{-\sigma q}(\sqrt{q+2}V(t)+1)\|\rho\|_{B^{\sigma}_{2,r}}\|\Delta_q\rho\|_{L^2}.
\end{split}
\end{equation}
Inserting \eqref{estimateII} and \eqref{estimateI} into \eqref{L2rho}, one can obtain
\begin{equation}
\label{delta_qrho}
\begin{split}
\frac{d}{dt}\|\Delta_q\rho(t)\|_{L^2}\leq \|\Delta_qf\|_{L^2}+
Cd_q2^{-\sigma q}(\sqrt{q+2}V(t)+1)\|\rho\|_{B^{\sigma}_{2,r}}.
\end{split}
\end{equation}
Denoting $s_t\triangleq s-\eta\int_0^tV(\tau)~d\tau$ for $t\in [0,T]$ with $\eta=\varepsilon\big(\int_0^TV(\tau)~d\tau\big)^{-1}$. Choosing $\sigma=s_t$ and integrating \eqref{delta_qrho} from $0$ to $t$ with respect to time variable and then multiplying by $2^{s_tq}$,
\begin{equation}
\label{rhoestimate}
\begin{split}
2^{s_tq}\|\Delta_q\rho(t)\|_{L^2}&\leq d_q \|\rho_0\|_{B^{s_t}_{2,r}}+d_q\int_0^t\|f(\tau)\|_{B^{s_t}_{2,1}}~d\tau\\&\quad+
Cd_q\int_0^t2^{\big(-\eta\int_{\tau}^{t}V(s)ds\big) q}(\sqrt{q+2}V(\tau)+1)\|\rho\|_{B^{s_\tau}_{2,r}}~d\tau.
\end{split}
\end{equation}
Choosing $q_0>0$ is the smallest integer such that
$$\frac{4C^2\|d_q\|_{\ell^r}^2}{(\log2)^2\eta^2}\leq q_0+2.$$
Then for $q\geq q_0$, we have
\begin{equation}
C\int_0^t2^{\big(-\eta\int_{\tau}^{t}V(s)ds\big) q}\sqrt{q+2}V(\tau)~d\tau\leq\frac{1}{2\|b_q\|_{\ell^r}}.
\end{equation}
Inserting these result into \eqref{rhoestimate} and taking $\ell^r$ norm of $q$, on can deduce
\begin{equation}
\label{rhoestimate1}
\begin{split}
\|\rho(t)\|_{B^{s_t}_{2,r}}&\leq C \|\rho_0\|_{B^{s}_{2,r}}+C\int_0^t\|f(\tau)\|_{B^{s}_{2,r}}~d\tau\\&\quad+
C\bigg(\sum_{q\geq q_0}\bigg(d_q\int_0^t2^{\big(-\eta\int_{\tau}^{t}V(s)ds\big) q}\sqrt{q+2}V(\tau)\|\rho\|_{B^{s_\tau}_{2,r}}~d\tau\bigg)^r\bigg)^{\frac1r}\\
&\quad+
C\bigg(\sum_{1\leq q<q_0 }\bigg(d_q\int_0^t2^{\big(-\eta\int_{\tau}^{t}V(s)ds\big) q}\sqrt{q+2}V(\tau)\|\rho\|_{B^{s_\tau}_{2,r}}~d\tau\bigg)^r\bigg)^{\frac1r}\\
&\leq C \|\rho_0\|_{B^{s}_{2,r}}+C\int_0^t\|f(\tau)\|_{B^{s}_{2,r}}~d\tau\\&\quad+
\frac12\sup_{t\in[0,T]}\|\rho\|_{B^{s_t}_{2,r}}+
C\sqrt{q_0+1}\int_0^tV(\tau)\|\rho\|_{B^{s_\tau}_{2,r}}~d\tau.\\
\end{split}
\end{equation}
Taking supremum of time $t$ from $0$ to $T$ and applying the Gr\"onwall's Lemma, we deduce
\begin{equation*}
\begin{split}
\sup_{t\in[0,T]}\|\rho(t)\|_{B^{s_t}_{2,r}}&\leq C(T) \bigg(\|\rho_0\|_{B^{s}_{2,r}}+\int_0^T\|f(\tau)\|_{B^{s}_{2,r}}~d\tau\bigg)e^{\sqrt{q_0+1}\int_0^TV(\tau)~d\tau}.
\end{split}
\end{equation*}
According to the definition of $q_0$, finally we obtain
\begin{equation*}
\begin{split}
\sup_{t\in[0,T]}\|\rho(t)\|_{B^{s_t}_{2,r}}&\leq C(T) \bigg(\|\rho_0\|_{B^{s}_{2,r}}+\int_0^T\|f(\tau)\|_{B^{s}_{2,r}}~d\tau\bigg)e^{\frac{C}{\eta}\int_0^TV(\tau)~d\tau},
\end{split}
\end{equation*}
which entails the desired inequality given that $s\geq s_t\geq s-\varepsilon$ for all $t\in[0,T]$.

\end{proof}

\end{appendices}

\vskip .4in
\section*{Acknowledgements}
M. Paicu is partially supported by the Agence Nationale de la Recherche, Project IFSMACS, grant ANR-15-CE40-0010. N. Zhu was partially supported by NSFC (No. 11771045, No. 11771043). Part of this work was done when N. Zhu was visiting Institut de Math\'ematiques de Bordeaux, and he would like to express the  his gratitude for providing him a very nice research environment.


\vskip .3in



\begin{thebibliography}{999}

\bibitem{ACSWXY}
D.~Adhikari, C.~Cao, H.~Shang, J.~Wu, X.~Xu, Z.~Ye, Global regularity results
for the 2d boussinesq equations with partial dissipation, {\it J. Differential Equations}, 260~(2) (2016) 1893--1917.

\bibitem{ACW1}
D.~Adhikari, C.~Cao, J.~Wu, The 2d boussinesq equations with vertical viscosity
and vertical diffusivity, {\it J. Differential Equations}, 249~(5) (2010),
1078--1088.

\bibitem{ACW2}
D.~Adhikari, C.~Cao, J.~Wu, Global regularity results for the 2d boussinesq
equations with vertical dissipation, {\it J. Differential Equations},
251~(6) (2011) 1637--1655.


\bibitem{AH}
H.~Abidi, T.~Hmidi, On the global well-posedness for boussinesq system, {\it J. Differential Equations} 233~(1) (2007), 199--220.
	
\bibitem{BCD}
H.~Bahouri, J.~Y. Chemin, R.~Danchin, {\it Fourier analysis and nonlinear partial differential equations}, Springer, (2011).
	
\bibitem{BC1}
A.~Bertozzi, P.~Constantin, Global regularity for vortex patches,
	{\it Comm. Math. Phys.} 152~(1) (1993), 19--28.

\bibitem{CD}
J.~Cannon, E.~DiBenedetto, The initial value problem for the boussinesq equations with data in $L^p$, in: {\it Approximation methods for Navier-Stokes
problems}, Springer, (1980), 129--144.
	
\bibitem{Chae1}
D.~Chae, Global regularity for the 2D boussinesq equations with partial
viscosity terms, {\it Adv. Math.} 203~(2) (2006), 497--513.
	
\bibitem{Chemin4}
J.-Y. Chemin, Persistance de structures g{\'e}om{\'e}triques dans les fluides
incompressibles bidimensionnels, {\it Ann. Sci. \'Ec. Norm. Sup\'er.} 26~(4) (1993), 517--542.
	
\bibitem{Chemin-perfect}
J.-Y. Chemin, {\it Perfect incompressible fluids}, Vol.~14, Oxford University Press,
(1998).

\bibitem{CDGG1}
J.~Y. Chemin, B.~Desjardins, I.~Gallagher, E.~Grenier,
{\it Mathematical geophysics. an introduction to rotating fluids and the navier-stokes equations}, Oxford Lecture, (2006).

\bibitem{CDGG2}
J.~Y. Chemin, B.~Desjardins, I.~Gallagher, E.~Grenier, Fluids with anisotropic
viscosity, {\it ESAIM Math. Model. Numer. Anal. } 34~(2) (2000),
315--335.

\bibitem{CW}
C.~Cao, J.~Wu, Global regularity for the two-dimensional anisotropic boussinesq
equations with vertical dissipation, {\it Arch. Ration. Mech. Anal.}, 208~(3) (2013), 985--1004.

\bibitem{Danchin1}
R.~Danchin, {\'E}volution temporelle d'une poche de tourbillon singuli{\`e}re,
{\it Comm. Partial Differential Equations} 22~(5-6) (1997), 685--721.
	
\bibitem{Danchin2}
R.~Danchin, Poches de tourbillon visqueuses,
{\it J. Math. Pures Appl.} 76~(7) (1997), 609--647.
	
\bibitem{DFP}
R.~Danchin, F.~Fanelli, M.~Paicu, A well-posedness result for viscous	compressible fluids with only bounded density, arXiv preprint
arXiv:1804.09503.
	
\bibitem{DM2}
R.~Danchin, P.~B. Mucha, The incompressible navier-stokes equations in vacuum,
{\it Comm. Pure Appl. Math.} (2018).
	
\bibitem{DM1}
R.~Danchin, P.~B. Mucha, A lagrangian approach for the incompressible
navier-stokes equations with variable density,{\it Comm. Pure Appl. Math.} 65~(10) (2012), 1458--1480.
	
\bibitem{DP3}
R.~Danchin, M.~Paicu, Les th{\'e}or{\`e}mes de leray et de fujita-kato pour le
syst{\`e}me de boussinesq partiellement visqueux, {\it Bull. Soc. Math. France} 136~(2) (2008), 261--309.
	
\bibitem{DP1} 
R. Danchin and M. Paicu, Global existence results for the anisotropic Boussinesq system in dimension two, {\it Math. Models Methods Appl. Sci.}  21 (2011), 421--457.	
	
\bibitem{DP2} 
R. Danchin and M. Paicu, Global well-posedness issues for the inviscid boussinesq system with Yudovich's
type data, {\it Comm. Math. Phys.} 290 (1) (2009), 1--14. 

\bibitem{DZ}
R.~Danchin, X.~Zhang, Global persistence of geometrical structures for the
boussinesq equation with no diffusion, {\it Comm. Partial Differential Equations}
42~(1) (2017), 68--99.
	
\bibitem{DZ2}
R.~Danchin, X.~Zhang, On the persistence of h\"older regular patches of density
for the inhomogeneous navier-stokes equations, {\it J. Ec. polytech. Math.}
4 (2017), 781--811.
	
\bibitem{Fanelli1}
F.~Fanelli, Conservation of geometric structures for non-homogeneous inviscid
incompressible fluids, {\it Comm. Partial Differential Equations}
37~(9) (2012), 1553--1595.
	
\bibitem{GS}
P.~Gamblin, X.~Saint-Raymond, On three-dimensional vortex patches, {\it Bull. Soc. Math. France} 123~(3) (1995), 375--424.
	
\bibitem{GG-J}
F.~Gancedo, E.~Garc{\'\i}a-Ju{\'a}rez, Global regularity for 2D boussinesq
temperature patches with no diffusion, {\it Ann. PDE} 3:14 (2017).
	
\bibitem{Guo}
B.~Guo, Spectral method for solving two-dimensional newton-boussinesq
equations, {\it Acta Math. Appl. Sin. Engl. Ser.} 5~(3) (1989),
208--218.
	
\bibitem{HH}
Z.~Hassainia, T.~Hmidi, On the inviscid boussinesq system with rough initial
data, {\it J. Math. Anal. Appl.} 430~(2) (2015),
777--809.
	
\bibitem{Hmidi1}
T.~Hmidi, R{\'e}gularit{\'e} h{\"o}ld{\'e}rienne des poches de tourbillon
visqueuses, {\it J. Math. Pures Appl.} 84~(11)
(2005), 1455--1495.
	
\bibitem{HK1}
T.~Hmidi, S.~Keraani, On the global well-posedness of the two-dimensional
boussinesq system with a zero diffusivity, {\it Adv. Differential Equations}
12~(4) (2007), 461.
	
\bibitem{HK2}
T.~Hmidi, S.~Keraani, On the global well-posedness of the boussinesq system
with zero viscosity, {\it to appear in Indiana Univ. Math. Journal.}
	
\bibitem{HL}
T.~Y. Hou, C.~Li, Global well-posedness of the viscous boussinesq equations,
{\it Discrete Contin. Dyn. Syst.} 12~(1) (2004), 1--12.

\bibitem{Iftimie}
D.~Iftimie, A uniqueness result for the navier--stokes equations with vanishing
vertical viscosity, {\it SIAM J. Math. Anal.} 33~(6) (2002),
1483--1493.

\bibitem{JL}
Q.~Jiu, J.~Liu, Global-wellposedness of 2D Boussinesq equations with mixed partial temperature-dependent viscosity and thermal diffusivity, {\it  Nonlinear Anal. TMA} 132 (2016),
227--239.

\bibitem{Kato}
 T. Kato. {\it Liapunov functions and monotonicity in the Euler and Navier-Stokes equations}, Lecture Notes in Mathematics 1450, Berlin: Springer-Verlag, 1990.	

\bibitem{KP}
T. Kato, G. Ponce. Commutator estimates and the Euler and Navier-Stokes
equations, {\it Comm. Pure Appl. Math.},  41  (1988): 891--907.
		
\bibitem{KPV}
C. Kenig, G. Ponce, L. Vega. Well-posedness of the initial value problem for the
Korteweg-de-Vries equation, {\it J. Amer. Math. Soc.},  4  (1991): 323--347.

\bibitem{LLT}
A. Larios, E. Lunasin, and E. S. Titi, Global well-posedness for the 2D Boussinesq system with anisotropic viscosity and without heat diffusion, {\it J. Differential Equations}
255~(9) (2013), 2636--2654.
	
\bibitem{LPZ1}
M.-J. Lai, R.~Pan, K.~Zhao, Initial boundary value problem for two-dimensional
viscous boussinesq equations, {\it Arch. Ration. Mech. Anal.}
199~(3) (2011), 739--760.
	
\bibitem{LPZ2}
H.~Li, R.~Pan, W.~Zhang, Initial boundary value problem for 2d boussinesq
equations with temperature-dependent diffusion, {\it J. Hyperbolic Differ. Equ.} 12~(03) (2015), 469--488.

\bibitem{LT1}
J. Li, E.S. Titi, Global well-posedness of the 2D Boussinesq equations with vertical dissipation, {\it Arch. Ration. Mech. Anal.} 220 (2016), 983--1001.
	
\bibitem{LZ1}
X.~Liao, P.~Zhang, On the global regularity of the two-dimensional density
patch for inhomogeneous incompressible viscous flow, {\it Arch. Ration. Mech. Anal.} 220~(3) (2016), 937--981.
	
\bibitem{LZ2}
X.~Liao, P.~Zhang, Global regularity of 2d density patches for viscous
inhomogeneous incompressible flow with general density: Low regularity case,
{\it Comm. Pure Appl. Math.} 72~(4) (2019),  835--884.

\bibitem{Paicu1}
M.~Paicu, Anisotropic Navier-Stokes equation in critical spaces,
Rev. Mat. Iber. 21~(1) (2005), 179--235.

\bibitem{PZ1}
M.~Paicu, P.~Zhang, Striated regularity of 2-d inhomogeneous incompressible
navier-stokes system with variable viscosity, arXiv preprint
arXiv:1711.04490.
	
\bibitem{Pedlosky}
J.~Pedlosky, {\it Geophysical fluid dynamics}, Springer Science \& Business Media,
(2013).
	
\bibitem{XZ}
X.~Xu, N.~Zhu, Global well-posedness for the 2D Boussinesq equations with partial temperature-dependent dissipative terms,
{\it  J. Math. Anal. Appl.} 466 (2018),  351--372.


	

	
	
	
\end{thebibliography}
\end{document}